\documentclass[11pt]{article}
\usepackage{mathrsfs}
\usepackage{amsfonts}
\usepackage{}
\usepackage{amsmath,amssymb,amsthm,latexsym,amstext}
\usepackage[mathscr]{eucal}
\usepackage{rotate,graphics,epsfig,epstopdf}
\usepackage{float}
\usepackage{color}
\usepackage{subfigure}
\usepackage{indentfirst}
\usepackage{bm}

\newcommand{\be}{\begin{eqnarray}}
\newcommand{\ee}{\end{eqnarray}}
\newcommand{\by}{\begin{eqnarray*}}
\newcommand{\ey}{\end{eqnarray*}}
\newcommand{\bn}{\begin{enumerate}}
\newcommand{\en}{\end{enumerate}}
\newcommand{\ei}{\end{itemize}}

\newtheorem{theorem}{Theorem}[section]
\newtheorem{lemma}[theorem]{Lemma}

\newtheorem{remark}[theorem]{Remark}

\renewcommand{\theequation}{\arabic{section}.\arabic{equation}}

\numberwithin{equation}{section}

\begin{document}
\date{}
\title{\bf The small mass limit for a McKean-Vlasov equation with state-dependent friction  \footnote{This work was supported by
the Natural Science Foundation of Jiangsu Province, BK20230899 and
the National Natural Science Foundation of China, 11771207.
}}
\author{ Chungang Shi$^{1}$\footnote{Corresponding author, shichungang@njust.edu.cn } 
\hskip1cm Mengmeng Wang$^{2}$\footnote{wmm3540@gmail.com}
\hskip1cm Yan Lv$^{3}$\footnote{lvyan@njust.edu.cn}
\hskip1cm Wei Wang$^{4}$\footnote{wangweinju@nju.edu.cn} \\
\texttt{{\scriptsize $^{1,3}$School of Mathematics and Statistics, Nanjing University of Science and Technology,
Nanjing, 210094, P. R. China}}\\
\texttt{{\scriptsize $^{2,4}$Department of Mathematics, Nanjing University,
Nanjing, 210023, P. R. China}}}\maketitle
\begin{abstract}
The small mass limit is derived for a McKean-Vlasov equation with state-dependent friction in $d$-dimensional space. By applying the averaging approach to a non-autonomous slow-fast system with the microscopic and macroscopic scales,  the convergence in distribution is obtained.
\end{abstract}



\section{Introduction}\label{sec:intro}
  \setcounter{equation}{0}
  \renewcommand{\theequation}
{1.\arabic{equation}}
In the presence of interaction and confinement potentials, the evolution of the classical Newton dynamics for indistinguishable $N$-point particle system with mass $\epsilon$ in $\mathbb{R}^{d}$ is given by:
\begin{eqnarray}\label{Intro:1.1}
\epsilon\ddot{x}_{i}+\gamma\dot{x}_{i}=-\nabla V(x_{i})-\frac{1}{N}\sum_{j=1}^{N}\nabla K(x_{i}-x_{j}), \quad i=1,2,\cdots,N,
\end{eqnarray}
where $x_{i}, \dot{x}_{i}\in\mathbb{R}^{d}$ are the position and velocity of the $i$-th particle respectively. The $\gamma>0$ describes the strength of linear damping in velocity and functions $V$ and $K:\mathbb{R}^{d}\to\mathbb{R}$ are the confinement potential and interaction potential respectively. As the number of particles $N$ tends to infinity, the microscopic descriptions of (\ref{Intro:1.1}) is approximately given by a McKean-Vlasov equation with the interaction potential replaced by a single averaged interaction~\cite{G},\cite{S},\cite{M},\cite{D}. 
In the real world, many systems may be disturbed by noise and damping may also depend on the state, in this case, the $N$-particle system (\ref{Intro:1.1}) is described by
\begin{eqnarray}\label{equ:main}
&&\epsilon\ddot{x}_{i}^{\epsilon}+\gamma(x_{i}^{\epsilon})\dot{x}_{i}^{\epsilon}=-\nabla V(x_{i}^{\epsilon})-\frac{1}{N}\sum_{j=1}^{N}\nabla K(x_{i}^{\epsilon}-x_{j}^{\epsilon})-\frac{1}{N}\sum_{j=1}^{N}\phi(x_{i}^{\epsilon}-x_{j}^{\epsilon})\dot{x}_{i}^{\epsilon}\nonumber\\
&&x_{i}^{\epsilon}(0)=x_{i,0},\quad \dot{x}_{i}^{\epsilon}(0)=v_{i,0},\quad\quad\quad\quad\quad\quad\quad\quad\quad +\sigma(x_{i}^{\epsilon})\dot{B}_{i}(t),
\end{eqnarray}
here $i=1,2,\cdots,N$, $\{B_{i}\}_{i=1}^{N}$ are independent Brownian motions. $\phi: \mathbb{R}^{d}\to \mathbb{R}^{d}\times\mathbb{R}^{d}$ denotes the nonlocal interaction~\cite{CS, HL, HT}. 
$\gamma: \mathbb{R}^{d}\to \mathbb{R}^{d}\times\mathbb{R}^{d}$ is the state-dependent friction. There are rich works on the stochastic $N$ particle system~\cite{CC},\cite{WLW},\cite{BC}. It is particularly worth mentioning that Wang et al.~\cite{WLW} considered a stochastic $N$ particle system with constant communication weight function and the small mass limit is derived by applying the averaging approach to distribution dependent slow-fast stochastic differential equations. Carrillo and Choi~\cite{CC} considered the more general $N$ particle system without stochastic perturbation forces and made use of a discrete version of a modulated kinetic energy together with the bounded Lipschitz distance for measures to derive the mean field limit. 
In this paper, we consider the following McKean-Vlasov stochastic differential equation
\begin{eqnarray}\label{equ:main1}
&&\epsilon\ddot{\bar{x}}_{i}^{\epsilon}+(\gamma(\bar{x}_{i}^{\epsilon})+\phi*\bar{\rho}_{t}^{\epsilon}(\bar{x}_{i}^{\epsilon}))\dot{\bar{x}}_{i}^{\epsilon}=-\nabla V(\bar{x}_{i}^{\epsilon})-\nabla K*\bar{\rho}_{t}^{\epsilon}(\bar{x}_{i}^{\epsilon})+\sigma(\bar{x}_{i}^{\epsilon})\dot{B}_{i}(t),\nonumber\\
&& \bar{x}_{i}^{\epsilon}(0)=\bar{x}_{i,0},\quad \dot{\bar{x}}_{i}^{\epsilon}(0)=\bar{v}_{i,0}.
\end{eqnarray}
where $\bar{\rho}_{t}^{\epsilon}$ is the distribution of $\bar{x}_{i}^{\epsilon}(t)$ on $\mathbb{R}^{d}$, $\nabla K*\bar{\rho}_{t}^{\epsilon}(x)=\int_{\mathbb{R}^{d}}\nabla K(x-y)\bar{\rho}_{t}^{\epsilon}(y)dy$ and $\phi*\bar{\rho}_{t}^{\epsilon}(x)=\int_{\mathbb{R}^{d}}\phi(x-y)\bar{\rho}_{t}^{\epsilon}(y)dy$. At the formal level, (\ref{equ:main1}) is the corresponding McKean-Vlasov stochastic differential equation of the stochastic $N$-particle systems (\ref{equ:main}) as $N\to\infty$~\cite{S, Se}.

Here we are concerned with the limit as $\epsilon\to0$ on both sides of (\ref{equ:main1}). For a single particle, if the particle has a very small mass, then the small mass limit is desirable. Especially, the limit is called the Smoluchowski-Kramers approximation with $K=0$ and $\phi=0$~\cite{M1, SMD, HMVW}. These works all consider the small mass limit on the microscopic scale. In recent work~\cite{WLW}, Wang et al. derived a small mass limit for a stochastic $N$-particle system on the macroscopic scale by applying the averaging approach. However, they only considered the case of constant friction coefficients. In the paper, we consider a small mass limit with a state-dependent friction coefficient on the macroscopic scale by introducing some macroscopic quantities, which may be regarded as a generalization of ~\cite{HMVW} to $N$-particle system. We mainly study the macroscopic scale limit of the Vlasov-Fokker-Planck equation corresponding to (\ref{equ:main1}) so as to obtain the small mass limit of the equation (\ref{equ:main1}) in the sense of distribution.

There are also some works on the macroscopic limit of Vlasov type equation. For example, Jabin~\cite{Ja} investigated the limit of some kinetic equation as the mass of the particles towards to zero. Karper~\cite{Ka} studied the hydrodynamic limit of a kinetic Cucker-Smale flocking model by means of a relative entropy method and the resulting asymptotic dynamics is an Euler-type flocking system. Huang~\cite{Hu} studied a kinetic Vlasov-Fokker-Planck equation and established a quantified estimate of the overdamped limit between the original equation and the aggregation-diffusion equation by adopting a probabilistic approach. For more results, we refer to~\cite{FS, FST, CT}.

The rest of the paper is organized as follows. Some essential preliminaries and main results are given in section \ref{Pre}. The proof of the small mass limit is presented in the last section.

\section{Preliminary and main result}\label{Pre}
  \setcounter{equation}{0}
  \renewcommand{\theequation}
{2.\arabic{equation}}
Let $(\Omega,\mathcal{F},\mathbb{P})$ be a complete probability space. $\mathbb{E}$ denotes expectation with respect to $\mathbb{P}$ and $\mathbb{E}^{j}$ denotes the expectation with respect to the distribution of $\bar{x}_{j}^{\epsilon}(t)$.
Assuming that $\|\cdot\|$ and $\langle\cdot,\cdot\rangle$ denotes the norm and product on $\mathbb{R}^{d}$. Let $A^{*}$ denotes the transpose of the matrix $A$.

Let $\mathcal{P}_{2}$ denote the sets consisting of Borel probability measure on $\mathbb{R}^{d}$ with 
\begin{equation*}
\int_{\mathbb{R}^{d}}\|x\|^{2}\mu(dx)<\infty,
\end{equation*}
for each $\mu\in\mathcal{P}_{2}$. For any $\mu$ and $\nu$ in $\mathcal{P}_{2}$, define the following Monge-Kantorovich (or Wasserstein) distance:
\begin{equation*}
dist_{MK,2}(\mu,\nu)=\inf_{\pi\in\Pi(\mu,\nu)}\Big[\int\int_{\mathbb{R}^{d}\times\mathbb{R}^{d}}\|x-y\|^{2}\pi(dx,dy)\Big]^{\frac{1}{2}},
\end{equation*}
where $\Pi(\mu,\nu)$ denotes the set of Borel probability measures $\Pi$ on $\mathbb{R}^{d}\times\mathbb{R}^{d}$ with the first and second marginals $\mu$ and $\nu$. Equivalently, for $\mu, \nu\in \mathcal{P}_{2}$,
\begin{equation}\label{equ:2.101}
dist_{MK,2}(\mu,\nu)=\inf_{(X,\bar{X})}[\mathbb{E}\|X-\bar{X}\|^{2}]^{\frac{1}{2}}
\end{equation}
with random variables $X$ and $\bar{X}$ in $\mathbb{R}^{d}$ having laws $\mu$ and $\nu$ respectively.

Next we give some assumptions to system (\ref{equ:main1}).

$(\mathbf{H_{1}})$. $\nabla V(\cdot), \nabla K(\cdot): \mathbb{R}^{d}\to\mathbb{R}^{d}$ are Lipschitz continuous with Lipschitz constant $L_{V}$ and $L_{K}$ respectively.

$(\mathbf{H_{2}})$. $\sigma(\cdot): \mathbb{R}^{d}\to\mathbb{R}^{d}\times\mathbb{R}^{d}$ is continuous differential and its derivative is bounded by $L_{\sigma}$.

$(\mathbf{H_{3}})$. $\gamma(\cdot): \mathbb{R}^{d}\to\mathbb{R}^{d}\times\mathbb{R}^{d}$ is continuous differential function with bounded derivative and its symmetric part $\frac{1}{2}(\gamma+\gamma*)$ has the smallest eigenvalue $\lambda_{1}(x)$ that is positive uniformly with respect to $x$,
\begin{equation*}
\lambda_{1}(x)\geq\lambda_{\gamma}>0.
\end{equation*} 

$(\mathbf{H_{4}})$. $\phi(\cdot): \mathbb{R}^{d}\to\mathbb{R}^{d}\times\mathbb{R}^{d}$ satisfies $(\mathbf{H_{3}})$, where $\lambda_{1}(\cdot), \lambda_{\gamma}$ are substituted with $\alpha_{1}(\cdot), \lambda_{\phi}$ respectively. Furthermore, there is a constant $L_{\phi}>0$ such that $\|\phi(\cdot)\|\leq L_{\phi}$.

The following lemma may be used to derive the solution of matrix equation in differential form.
\begin{lemma}{\rm\cite[Theorem 2]{BBH}}\label{YT}
Let $I\subset\mathbb{R}$ an open interval with $t_{0}\in I$, $A\in \mathbb{C}^{n\times n}, B\in \mathbb{C}^{m\times m}, C\in \mathcal{C}(I,\mathbb{C}^{n\times n})$ and $D\in \mathbb{C}^{m\times n}$. The differential Sylvester equation
\begin{eqnarray*}
\frac{d}{dt}X(t)=AX(t)+X(t)B+C(t),\quad X(t_{0})=D,
\end{eqnarray*}
has the unique solution
\begin{eqnarray*}
X(t)=e^{A(t-t_{0})}De^{B(t-t_{0})}+\int_{t_{0}}^{t}e^{A(t-s)}C(s)e^{B(t-s)}ds.
\end{eqnarray*}
\end{lemma}

\begin{lemma}{\rm\cite[p120]{Ke}}\label{YL}
Let $A=(a_{ij})_{1\leq i,j\leq d}$ and $u=(u_{i})_{1\leq i\leq d}$ be $d\times d$ matrix and $d\times1$ vector respectively. Each element of $A$ and $u$ is a function of $x=(x_{1},x_{2},\dots,x_{d})$, then
\begin{eqnarray*}
\nabla\cdot(Au)=(\nabla\cdot A)u+{\rm Tr}(A~{\rm grad}~u),
\end{eqnarray*}
where ${\rm grad}~u=(\frac{\partial u_{i}}{\partial x_{j}})_{1\leq i,j\leq d}$.
\end{lemma}

Then we give the main result of the paper.

\begin{theorem}\label{main2}
Under the assumptions $(\mathbf{H_{1}})$-$(\mathbf{H_{4}})$, for every $T>0$ and any fixed $t_{*}>0$, $\bar{\rho}_{t}^{\epsilon}$ converges weakly to $\bar{\rho}_{t}$ for $t_{*}\leq t\leq T$ as $\epsilon\to0$ with
\begin{eqnarray}
\partial_{t}\bar{\rho}_{t}(x)&=&-\nabla_{x}\cdot[(\gamma(x)+\phi*\bar{\rho}_{t}(x))^{-1}(\nabla V(x)+\nabla K*\bar{\rho}_{t}(x))\bar{\rho}_{t}(x)\nonumber\\
&&+(\gamma(x)+\phi*\bar{\rho}_{t}(x))^{-1}\nabla_{x}\cdot(\bar{\rho}_{t}(x)J(x,\bar{\rho}_{t}(x)))]\nonumber\\
\end{eqnarray}
which is the Vlasov-Fokker-Planck equation corresponding to the following SDE
\begin{eqnarray}\label{main-equ:2.2}
&&dx(t)=\big[(\gamma(x(t))+\tilde{E}\phi(x(t)-\tilde{x}(t)))^{-1}(\nabla V(x(t))+\tilde{E}\nabla K(x(t)-\tilde{x}(t))
\nonumber\\
&&+S(x(t))+(\gamma(x(t))+\tilde{E}\phi(x(t)-\tilde{x}(t)))^{-1}\sigma(x(t))dB_{t},
\end{eqnarray}
where $\tilde{x}(t)$ is a version of $x(t)$ and $\tilde{E}$ is the expectation with respect to the distribution of $\tilde{x}(t)$ and 
\begin{eqnarray*}
S_{i}(x)=\frac{\partial}{\partial x_{k}}((\gamma(x)+\phi*\bar{\rho}_{t}(x))^{-1})_{ij}J_{jk}(x,\bar{\rho}_{t}(x)).
\end{eqnarray*}
Furthermore, $J(x,\bar{\rho}_{t}(x))$ is the solution of the Lyapunov equation
\begin{eqnarray}\label{JES}
J(x,\bar{\rho}_{t}(x))(\gamma(x(t))&+&\tilde{E}\phi(x(t)-\tilde{x}(t)))^{*}\nonumber\\
&+&(\gamma(x(t))+\tilde{E}\phi(x(t)-\tilde{x}(t)))J(x,\bar{\rho}_{t}(x))=\sigma(x)\sigma(x)^{*},\nonumber\\
\end{eqnarray}
and $\bar{x}_{i}^{\epsilon}(t)$ converges  in distribution to $x(t)$ in $C(0,T;\mathbb{R}^{d})$. 
\end{theorem}
\begin{remark}
By~\cite[p179]{BE}, the equation (\ref{JES}) has a unique solution
\begin{eqnarray*}
J(x,\bar{\rho}_{t}(x))=\int_{0}^{\infty}e^{-(\gamma(x)+\phi*\bar{\rho}_{t}(x))s}\sigma(x)\sigma(x)^{*}e^{-(\gamma(x)+\phi*\bar{\rho}_{t}(x))^{*}s}ds.
\end{eqnarray*}
\end{remark}

\section{An averaging result: $\epsilon\to0$}\label{Aver}
  \setcounter{equation}{0}
  \renewcommand{\theequation}
{3.\arabic{equation}}
We need some uniform estimates on solutions $\bar{x}_{i}^{\epsilon}$.
\begin{lemma}\label{SUB}
Under assumptions $(\mathbf{H_{1}})$-$(\mathbf{H_{4}})$, for each $T>0$, there exists $C_{T}>0$ such that 
\begin{eqnarray*}
\mathbb{E}\sup_{0\leq t\leq T}\|\bar{x}_{i}^{\epsilon}(t)\|^{2}
\leq C_{T}(1+\mathbb{E}\|\bar{x}_{i,0}\|^{2}+\mathbb{E}\|\bar{v}_{i,0}\|^{2}),
\end{eqnarray*}
where $C_{T}$ also depends on parameters $\lambda_{\gamma},\lambda_{\phi},L_{\sigma},L_{K},L_{V},\|\nabla K(0)\|$.
\end{lemma}
\begin{proof}
Rewrite equation (\ref{equ:main1}) as
\begin{eqnarray}\label{EWJV}
d\dot{\bar{x}}_{i}^{\epsilon}+\frac{1}{\epsilon}(\gamma(\bar{x}_{i}^{\epsilon})+\phi*\bar{\rho}_{t}^{\epsilon}(\bar{x}_{i}^{\epsilon}))\dot{\bar{x}}_{i}^{\epsilon}=-\frac{1}{\epsilon}\nabla V(\bar{x}_{i}^{\epsilon})-\frac{1}{\epsilon}\nabla K*\bar{\rho}_{t}^{\epsilon}(\bar{x}_{i}^{\epsilon})+\frac{\sigma(\bar{x}_{i}^{\epsilon})}{\epsilon}dB_{i}(t).
\end{eqnarray}
Introduce process $y_{i}^{\epsilon}(t)$ satisfying
\begin{eqnarray}\label{equ:2.2}
\frac{dy_{i}^{\epsilon}(t)}{dt}=-\frac{1}{\epsilon}(\gamma(\bar{x}_{i}^{\epsilon})+\phi*\bar{\rho}_{t}^{\epsilon}(\bar{x}_{i}^{\epsilon}))y_{i}^{\epsilon}(t), \quad y_{i}^{\epsilon}(0)=I_{d},
\end{eqnarray}
where $I_{d}$ denotes the $d$-dimensional unit vector.
Then
\begin{eqnarray*}\label{equ:2.3}
d(y_{i}^{\epsilon}(t)^{-1}\dot{\bar{x}}_{i}^{\epsilon}(t))&=&-\frac{1}{\epsilon}y_{i}^{\epsilon}(t)^{-1}\nabla V(\bar{x}_{i}^{\epsilon})-\frac{1}{\epsilon}y_{i}^{\epsilon}(t)^{-1}\nabla K*\bar{\rho}_{t}^{\epsilon}(\bar{x}_{i}^{\epsilon})\\
&&+\frac{\sigma(\bar{x}_{i}^{\epsilon})}{\epsilon}y_{i}^{\epsilon}(t)^{-1}dB_{i}(t).
\end{eqnarray*}
Integrating from $0$ to $t$ and multiplying $y_{i}^{\epsilon}(t)$ on both sides yields
\begin{eqnarray}\label{equ:2.4}
\dot{\bar{x}}_{i}^{\epsilon}(t)&=&y_{i}^{\epsilon}(t)\bar{v}_{i,0}-\frac{1}{\epsilon}\int_{0}^{t}y_{i}^{\epsilon}(t)y_{i}^{\epsilon}(s)^{-1}\nabla V(\bar{x}_{i}^{\epsilon}(s))ds\nonumber\\
&&-\frac{1}{\epsilon}\int_{0}^{t}y_{i}^{\epsilon}(t)y_{i}^{\epsilon}(s)^{-1}\nabla K*\bar{\rho}_{s}^{\epsilon}(\bar{x}_{i}^{\epsilon}(s))ds\nonumber\\
&&+\frac{1}{\epsilon}\int_{0}^{t}y_{i}^{\epsilon}(t)y_{i}^{\epsilon}(s)^{-1}\sigma(\bar{x}_{i}^{\epsilon}(s))dB_{i}(s),
\end{eqnarray}
then
\begin{eqnarray}\label{equ:2.5}
\bar{x}_{i}^{\epsilon}(t)&=&\bar{x}_{i,0}+\int_{0}^{t}y_{i}^{\epsilon}(s)\bar{v}_{i,0}ds-\frac{1}{\epsilon}\int_{0}^{t}\int_{0}^{s}y_{i}^{\epsilon}(s)y_{i}^{\epsilon}(u)^{-1}\nabla V(\bar{x}_{i}^{\epsilon}(u))duds\nonumber\\
&&-\frac{1}{\epsilon}\int_{0}^{t}\int_{0}^{s}y_{i}^{\epsilon}(s)y_{i}^{\epsilon}(u)^{-1}\nabla K*\bar{\rho}_{u}^{\epsilon}(\bar{x}_{i}^{\epsilon}(u))duds\nonumber\\
&&+\frac{1}{\epsilon}\int_{0}^{t}\int_{0}^{s}y_{i}^{\epsilon}(s)y_{i}^{\epsilon}(u)^{-1}\sigma(\bar{x}_{i}^{\epsilon}(u))dB_{i}(u)ds\nonumber\\
&\triangleq& \bar{x}_{i,0}+\sum_{j=1}^{3} I_{j}^{\epsilon}(t).
\end{eqnarray}
Let 
\begin{equation*}
\psi^{\epsilon}(t)=y_{i}^{\epsilon}(t)y_{i}^{\epsilon}(s)^{-1}(-\nabla V(\bar{x}_{i}^{\epsilon}(s))-\nabla K*\bar{\rho}_{s}^{\epsilon}(\bar{x}_{i}^{\epsilon}(s))),
\end{equation*}
then 
\begin{equation*}
\psi^{\epsilon}(s)=-\nabla V(\bar{x}_{i}^{\epsilon}(s))-\nabla K*\bar{\rho}_{s}^{\epsilon}(\bar{x}_{i}^{\epsilon}(s)),
\end{equation*}
and by (\ref{equ:2.2}),
\begin{eqnarray*}
\frac{d\psi^{\epsilon}(t)}{dt}&=&\frac{dy_{i}^{\epsilon}(t)}{dt}y_{i}^{\epsilon}(s)^{-1}(-\nabla V(\bar{x}_{i}^{\epsilon}(s))-\nabla K*\bar{\rho}_{s}^{\epsilon}(\bar{x}_{i}^{\epsilon}(s)))\\
&=&-\frac{1}{\epsilon}(\gamma(\bar{x}_{i}^{\epsilon}(t))+\phi*\bar{\rho}_{t}^{\epsilon}(\bar{x}_{i}^{\epsilon}(t)))\psi^{\epsilon}(t).
\end{eqnarray*}
Thus we have~\cite[Lemma 4.2 of Chapter IV]{HA}
\begin{eqnarray*}\label{equ:2.6}
\psi^{\epsilon}(t)&=&\psi^{\epsilon}(s)e^{-\frac{1}{\epsilon}\int_{s}^{t}(\gamma(\bar{x}_{i}^{\epsilon}(u))+\phi*\bar{\rho}_{u}^{\epsilon}(\bar{x}_{i}^{\epsilon}(u)))du}\nonumber\\
&=&(-\nabla V(\bar{x}_{i}^{\epsilon}(s))-\nabla K*\bar{\rho}_{s}^{\epsilon}(\bar{x}_{i}^{\epsilon}(s)))e^{-\frac{1}{\epsilon}\int_{s}^{t}(\gamma(\bar{x}_{i}^{\epsilon}(u))+\phi*\bar{\rho}_{u}^{\epsilon}(\bar{x}_{i}^{\epsilon}(u)))du},
\end{eqnarray*}
and by (\ref{equ:2.2}),
\begin{eqnarray}\label{equ:2.7}
y_{i}^{\epsilon}(t)=y_{i}^{\epsilon}(0)e^{-\frac{1}{\epsilon}\int_{s}^{t}(\gamma(\bar{x}_{i}^{\epsilon}(u))+\phi*\bar{\rho}_{u}^{\epsilon}(\bar{x}_{i}^{\epsilon}(u)))du}.
\end{eqnarray}
Then by $(\mathbf{H_{3}})$ and $(\mathbf{H_{4}})$,
\begin{eqnarray}\label{equ:2.61}
&&\|\psi^{\epsilon}(t)\|\nonumber\\
&&\leq \|\nabla V(\bar{x}_{i}^{\epsilon}(s))+\nabla K*\bar{\rho}_{s}^{\epsilon}(\bar{x}_{i}^{\epsilon}(s))e^{-\frac{1}{\epsilon}\int_{s}^{t}(\gamma(\bar{x}_{i}^{\epsilon}(u))+\phi*\bar{\rho}_{u}^{\epsilon}(\bar{x}_{i}^{\epsilon}(u)))du}\|\nonumber\\
&&\leq\|\nabla V(\bar{x}_{i}^{\epsilon}(s))+\nabla K*\bar{\rho}_{s}^{\epsilon}(\bar{x}_{i}^{\epsilon}(s))e^{-\frac{1}{\epsilon}\int_{s}^{t}(\gamma(\bar{x}_{i}^{\epsilon}(u))+\mathbb{E}^{j}\phi(\bar{x}_{i}^{\epsilon}(u)-\bar{x}_{j}^{\epsilon}(u)))du}\|\nonumber\\
&&\leq\|\nabla V(\bar{x}_{i}^{\epsilon}(s))+\nabla K*\bar{\rho}_{s}^{\epsilon}(\bar{x}_{i}^{\epsilon}(s))\|e^{-\frac{1}{\epsilon}(\lambda_{\gamma}+\lambda_{\phi})(t-s)},
\end{eqnarray}
and
\begin{eqnarray*}\label{equ:2.8}
\|y_{i}^{\epsilon}(t)\bar{v}_{i,0}\|\leq \|\bar{v}_{i,0}\|e^{-\frac{1}{\epsilon}(\lambda_{\gamma}+\lambda_{\phi})t}.
\end{eqnarray*}
Furthermore,
\begin{eqnarray*}\label{equ:2.91}
\|y_{i}^{\epsilon}(t)y_{i}^{\epsilon}(s)^{-1}\|\leq e^{-\frac{1}{\epsilon}(\lambda_{\gamma}+\lambda_{\phi})(t-s)},
\end{eqnarray*}
and
\begin{eqnarray}\label{equ:2.15}
\Big\|\frac{1}{\epsilon}\int_{u}^{t}y_{i}^{\epsilon}(s)y_{i}^{\epsilon}(u)^{-1}ds\Big\|\leq\frac{1}{\lambda_{\gamma}+\lambda_{\phi}}.
\end{eqnarray}
Then
\begin{eqnarray}\label{equ:2.8}
\|I_{1}^{\epsilon}(t)\|&\leq& \int_{0}^{t}\|y_{i}^{\epsilon}(s)\bar{v}_{i,0}\|ds\leq \|\bar{v}_{i,0}\|\int_{0}^{t}e^{-\frac{1}{\epsilon}(\lambda_{\gamma}+\lambda_{\phi})s}ds\nonumber\\
&\leq& \frac{\epsilon}{\lambda_{\gamma}+\lambda_{\phi}}\|\bar{v}_{i,0}\|,
\end{eqnarray}
and by (\ref{equ:2.61}),
\begin{eqnarray}\label{equ:2.9}
\|I_{2}^{\epsilon}(t)\|&\leq& \frac{1}{\epsilon}\int_{0}^{t}\int_{0}^{s}\|y_{i}^{\epsilon}(s)y_{i}^{\epsilon}(u)^{-1}(\nabla V(\bar{x}_{i}^{\epsilon}(u))+\nabla K*\bar{\rho}_{u}^{\epsilon}(\bar{x}_{i}^{\epsilon}(u)))\|duds\nonumber\\
&\leq&\frac{1}{\epsilon}\int_{0}^{t}\int_{0}^{s}\|\nabla V(\bar{x}_{i}^{\epsilon}(u))+\nabla K*\bar{\rho}_{u}^{\epsilon}(\bar{x}_{i}^{\epsilon}(u))\|e^{-\frac{\lambda_{\gamma}+\lambda_{\phi}}{\epsilon}(s-u)}duds\nonumber\\
&\leq&\frac{1}{\epsilon}\int_{0}^{t}\|\nabla V(\bar{x}_{i}^{\epsilon}(u))+\nabla K*\bar{\rho}_{u}^{\epsilon}(\bar{x}_{i}^{\epsilon}(u))\|\int_{u}^{t}e^{-\frac{\lambda_{\gamma}+\lambda_{\phi}}{\epsilon}(s-u)}dsdu\nonumber\\
&\leq&\frac{1}{\lambda_{\gamma}+\lambda_{\phi}}\int_{0}^{t}\|\nabla V(\bar{x}_{i}^{\epsilon}(u))+\nabla K*\bar{\rho}_{u}^{\epsilon}(\bar{x}_{i}^{\epsilon}(u))\|du.
\end{eqnarray}
Note that, by ($\mathbf{H}_{1}$),
\begin{eqnarray}\label{equ:2.10}
\|\nabla V(\bar{x}_{i}^{\epsilon}(u))\|=\|\nabla V(\bar{x}_{i}^{\epsilon}(u))-\nabla V(0)\|\leq L_{V}\|\bar{x}_{i}^{\epsilon}(u)\|,
\end{eqnarray}
and 
\begin{eqnarray}\label{equ:2.11}
&&\|(\nabla K*\bar{\rho}_{u}^{\epsilon})(\bar{x}_{i}^{\epsilon}(u))\|=\|\mathbb{E}^{j}\nabla K(\bar{x}_{i}^{\epsilon}(u)-\bar{x}_{j}^{\epsilon}(u))\|\nonumber\\
&&\leq \mathbb{E}^{j}\|\nabla K(\bar{x}_{i}^{\epsilon}(u)-\bar{x}_{j}^{\epsilon}(u))-\nabla K(0)\|+\|\nabla K(0)\|\nonumber\\
&&\leq L_{K}(\|\bar{x}_{i}^{\epsilon}(u)\|+\mathbb{E}\|\bar{x}_{i}^{\epsilon}(u)\|)+\|\nabla K(0)\|\nonumber\\
&&\leq C_{L_{K},\|\nabla K(0)\|}(1+\|\bar{x}_{i}^{\epsilon}(u)\|+\mathbb{E}\|\bar{x}_{i}^{\epsilon}(u)\|).
\end{eqnarray}
Then from (\ref{equ:2.9}) to (\ref{equ:2.11}), we have
\begin{eqnarray}\label{equ:2.12}
\|I_{2}^{\epsilon}(t)\|\leq C_{\lambda_{\gamma},\lambda_{\phi},L_{K},L_{V},\|\nabla K(0)\|}\int_{0}^{t}(1+\|\bar{x}_{i}^{\epsilon}(u)\|+\mathbb{E}\|\bar{x}_{i}^{\epsilon}(u)\|)du,
\end{eqnarray}
and 
\begin{eqnarray}\label{equ:2.13}
\mathbb{E}\sup_{0\leq t\leq T}\|I_{2}^{\epsilon}(t)\|^{2}&\leq&C_{T,\lambda_{\gamma},\lambda_{\phi},L_{K},L_{V},\|\nabla K(0)\|}\nonumber\\
&&+C_{T,\lambda_{\gamma},\lambda_{\phi},L_{K},L_{V},\|\nabla K(0)\|}\int_{0}^{T}\mathbb{E}\sup_{0\leq s\leq t}\|\bar{x}_{i}^{\epsilon}(s)\|^{2}dt.\nonumber\\
\end{eqnarray}
For $I_{3}^{\epsilon}(t)$, by stochastic Fubini's theorem, (\ref{equ:2.7}), (\ref{equ:2.15}) and the Burkh${\rm \ddot{o}}$lder-Davies-Gundy inequality,
\begin{eqnarray}\label{equ:2.14}
\mathbb{E}\sup_{0\leq t\leq T}\|I_{3}^{\epsilon}(t)\|^{2}&=&\mathbb{E}\sup_{0\leq t\leq T}\Big\|\frac{1}{\epsilon}\int_{0}^{t}\int_{0}^{s}y_{i}^{\epsilon}(s)y_{i}^{\epsilon}(u)^{-1}\sigma(\bar{x}_{i}^{\epsilon}(u))dB_{i}(u)ds\Big\|^{2}\nonumber\\
&=&\mathbb{E}\sup_{0\leq t\leq T}\Big\|\frac{1}{\epsilon}\int_{0}^{t}\int_{u}^{t}y_{i}^{\epsilon}(s)y_{i}^{\epsilon}(u)^{-1}\sigma(\bar{x}_{i}^{\epsilon}(u))dsdB_{i}(u)\Big\|^{2}\nonumber\\
&\leq&\frac{4}{\epsilon^{2}}\int_{0}^{T}\mathbb{E}\Big(\int_{u}^{t}e^{-\frac{\lambda_{\gamma}+\lambda_{\phi}}{\epsilon}(s-u)}\|\sigma(\bar{x}_{i}^{\epsilon}(u))\|ds\Big)^{2}du\nonumber\\
&\leq& C_{\lambda_{\gamma},\lambda_{\phi},L_{\sigma},T}\Big(1+\int_{0}^{T}\mathbb{E}\sup_{0\leq s\leq t}\|\bar{x}_{i}^{\epsilon}(s)\|^{2}dt\Big).
\end{eqnarray}
Now by (\ref{equ:2.5}), (\ref{equ:2.8}), (\ref{equ:2.13}) and (\ref{equ:2.14}),
\begin{eqnarray*}\label{equ:2.17}
\mathbb{E}\sup_{0\leq t\leq T}\|\bar{x}_{i}^{\epsilon}(t)\|^{2}&\leq&C_{T,L_{\sigma},\lambda_{\gamma},\lambda_{\phi},L_{K},L_{V},\|\nabla K(0)\|}(1+\mathbb{E}\|\bar{x}_{i,0}\|^{2}+\mathbb{E}\|\bar{v}_{i,0}\|^{2})\nonumber\\
&&+C_{T,\lambda_{\gamma},\lambda_{\phi},L_{K},L_{V},\|\nabla K(0)\|}\int_{0}^{T}\mathbb{E}\sup_{0\leq s\leq t}\|\bar{x}_{i}^{\epsilon}(s)\|^{2}dt,
\end{eqnarray*}
and the Gronwall inequality yields
\begin{eqnarray*}\label{equ:2.18}
\mathbb{E}\sup_{0\leq t\leq T}\|\bar{x}_{i}^{\epsilon}(t)\|^{2}\leq C_{T,\lambda_{\gamma},\lambda_{\phi},L_{K},L_{V},\|\nabla K(0)\|}(1+\mathbb{E}\|\bar{x}_{i,0}\|^{2}+\mathbb{E}\|\bar{v}_{i,0}\|^{2}).
\end{eqnarray*}
\end{proof}

Next we show that $\bar{x}_{i}^{\epsilon}(t)$ is H${\rm\ddot{o}}$lder continuous.

\begin{lemma}\label{SC}
Suppose that {\rm($\mathbf{H}_{1}$)}-{\rm($\mathbf{H}_{4}$)} holds, then for any $0\leq s,t\leq T$, there exists a constant $C_{T}>0$ such that
\begin{eqnarray*}
\mathbb{E}\|\bar{x}_{i}^{\epsilon}(t)-\bar{x}_{i}^{\epsilon}(s)\|^{2}\leq C_{T}|t-s|,
\end{eqnarray*}
where $C_{T}$ also depends on parameters $\lambda_{\gamma},\lambda_{\phi},L_{\sigma},L_{K},L_{V}$ and initial value.
\end{lemma}
\begin{proof}
For any $0\leq s\leq t\leq T$, by (\ref{equ:2.4}), 
\begin{eqnarray}\label{RZo}
&&\mathbb{E}\|\bar{x}_{i}^{\epsilon}(t)-\bar{x}_{i}^{\epsilon}(s)\|^{2}=\mathbb{E}\Big\|\int_{s}^{t}\dot{\bar{x}}_{i}^{\epsilon}(u)du\Big\|^{2}\nonumber\\
&&=\mathbb{E}\Big\|\int_{s}^{t}\Big[y_{i}^{\epsilon}(u)\bar{v}_{i,0}-\frac{1}{\epsilon}\int_{0}^{u}y_{i}^{\epsilon}(u)y_{i}^{\epsilon}(r)^{-1}\nabla V(\bar{x}_{i}^{\epsilon}(r))dr\nonumber\\
&&\quad-\frac{1}{\epsilon}\int_{0}^{u}y_{i}^{\epsilon}(u)y_{i}^{\epsilon}(r)^{-1}\nabla K*\bar{\rho}_{r}^{\epsilon}(\bar{x}_{i}^{\epsilon}(r))dr+\frac{1}{\epsilon}\int_{0}^{u}y_{i}^{\epsilon}(u)y_{i}^{\epsilon}(r)^{-1}\sigma(\bar{x}_{i}^{\epsilon}(r))dB_{i}(r)\Big]du\Big\|^{2}\nonumber\\
&&\leq4\mathbb{E}\Big\|\int_{s}^{t}y_{i}^{\epsilon}(u)\bar{v}_{i,0}\Big\|^{2}+\frac{4}{\epsilon^{2}}\mathbb{E}\Big\|\int_{s}^{t}\int_{0}^{u}y_{i}^{\epsilon}(u)y_{i}^{\epsilon}(r)^{-1}\nabla V(\bar{x}_{i}^{\epsilon}(r))drdu\Big\|^{2}\nonumber\\
&&\quad+\frac{4}{\epsilon^{2}}\mathbb{E}\Big\|\int_{s}^{t}\int_{0}^{u}y_{i}^{\epsilon}(u)y_{i}^{\epsilon}(r)^{-1}\nabla K*\bar{\rho}_{r}^{\epsilon}(\bar{x}_{i}^{\epsilon}(r))drdu\Big\|^{2}\nonumber\\
&&\quad+\frac{8}{\epsilon^{2}}\Big\|\int_{s}^{t}\int_{0}^{u}y_{i}^{\epsilon}(u)y_{i}^{\epsilon}(r)^{-1}\sigma(\bar{x}_{i}^{\epsilon}(r))dB_{i}(r)du\Big\|^{2}\nonumber\\
&&\triangleq \sum_{j=1}^{4}R_{j}^{\epsilon}(t).
\end{eqnarray}
Next we estimate the four terms respectively.
For $R_{1}^{\epsilon}(t)$, by (\ref{equ:2.7}), the H${\rm\ddot{o}}$lder inequality and  {\rm($\mathbf{H}_{3}$)}, we have
\begin{eqnarray}\label{R1}
R_{1}^{\epsilon}(t)&=&4\mathbb{E}\Big\|\int_{s}^{t}y_{i}^{\epsilon}(u)\bar{v}_{i,0}du\Big\|^{2}\leq 4|t-s|\int_{s}^{t}\mathbb{E}\|y_{i}^{\epsilon}(u)\bar{v}_{i,0}\|^{2}du\nonumber\\
&\leq&4|t-s|\int_{s}^{t}e^{-\frac{2}{\epsilon}(\lambda_{\gamma}+\lambda_{\phi})u}du\cdot\mathbb{E}\|\bar{v}_{i,0}\|^{2}\nonumber\\
&\leq&C_{\lambda_{\gamma},\lambda_{\phi},T}|t-s|\mathbb{E}\|\bar{v}_{i,0}\|^{2}.
\end{eqnarray}
For $R_{2}^{\epsilon}(t)$, by {\rm($\mathbf{H}_{3}$)},
\begin{eqnarray}\label{R2}
R_{2}^{\epsilon}(t)&=&\frac{4}{\epsilon^{2}}\mathbb{E}\Big\|\int_{s}^{t}\int_{0}^{u}y_{i}^{\epsilon}(u)y_{i}^{\epsilon}(r)^{-1}\nabla V(\bar{x}_{i}^{\epsilon}(r))drdu\Big\|^{2}\nonumber\\
&\leq&\frac{4}{\epsilon^{2}}\mathbb{E}\Big(\int_{s}^{t}\int_{0}^{u}\|y_{i}^{\epsilon}(u)y_{i}^{\epsilon}(r)^{-1}\|\|\nabla V(\bar{x}_{i}^{\epsilon}(r))\|drdu\Big)^{2}\nonumber\\
&\leq&\frac{8L_{V}^{2}}{\epsilon^{2}}\mathbb{E}\Big(\int_{s}^{t}\int_{0}^{u}\|y_{i}^{\epsilon}(u)y_{i}^{\epsilon}(r)^{-1}\|drdu\Big)^{2}\nonumber\\
&&+\frac{8L_{V}^{2}}{\epsilon^{2}}\mathbb{E}\Big(\int_{s}^{t}\int_{0}^{u}\|y_{i}^{\epsilon}(u)y_{i}^{\epsilon}(r)^{-1}\|\|\bar{x}_{i}^{\epsilon}(r)\|drdu\Big)^{2}\nonumber\\
&\triangleq&R_{2,1}^{\epsilon}(t)+R_{2,2}^{\epsilon}(t).
\end{eqnarray}
By {\rm($\mathbf{H}_{3}$)} and the H${\rm\ddot{o}}$lder inequality,
\begin{eqnarray}\label{R2-1}
R_{2,1}^{\epsilon}(t)&=&\frac{8L_{V}^{2}}{\epsilon^{2}}\mathbb{E}\Big(\int_{s}^{t}\int_{0}^{u}\|y_{i}^{\epsilon}(u)y_{i}^{\epsilon}(r)^{-1}\|drdu\Big)^{2}\nonumber\\
&\leq&\frac{8L_{V}^{2}}{\epsilon^{2}}|t-s|\mathbb{E}\int_{s}^{t}\Big(\int_{0}^{u}e^{-\frac{2}{\epsilon}(\lambda_{\gamma}+\lambda_{\phi})(u-r)}dr\Big)^{2}du\nonumber\\
&\leq& C_{\lambda_{\gamma},\lambda_{\phi},L_{V},T}|t-s|.
\end{eqnarray}
Similarly, together with Lemma \ref{SUB},
\begin{eqnarray}\label{R2-2}
R_{2,2}^{\epsilon}(t)&=&\frac{8L_{V}^{2}}{\epsilon^{2}}\mathbb{E}\Big(\int_{s}^{t}\int_{0}^{u}\|y_{i}^{\epsilon}(u)y_{i}^{\epsilon}(r)^{-1}\|\|\bar{x}_{i}^{\epsilon}(r)\|drdu\Big)^{2}\nonumber\\
&\leq&\frac{8L_{V}^{2}}{\epsilon^{2}}\mathbb{E}\sup_{0\leq t\leq T}\|\bar{x}_{i}^{\epsilon}(t)\|^{2}\Big(\int_{s}^{t}\int_{0}^{u}e^{-\frac{2}{\epsilon}(\lambda_{\gamma}+\lambda_{\phi})(u-r)}drdu\Big)^{2}\nonumber\\
&\leq&C_{\lambda_{\gamma},\lambda_{\phi},L_{V},T}|t-s|(1+\mathbb{E}\|\bar{x}_{i,0}\|^{2}+\mathbb{E}\|\bar{v}_{i,0}\|^{2}).
\end{eqnarray}
Then combining (\ref{R2}) with (\ref{R2-1}) and (\ref{R2-2}),
\begin{eqnarray}\label{R2R}
R_{2}^{\epsilon}(t)\leq C_{\lambda_{\gamma},\lambda_{\phi},L_{V},T}|t-s|(1+\mathbb{E}\|\bar{x}_{i,0}\|^{2}+\mathbb{E}\|\bar{v}_{i,0}\|^{2}).
\end{eqnarray}
For $R_{3}^{\epsilon}(t)$, by (\ref{equ:2.7}), {\rm($\mathbf{H}_{3}$)} and Lemma \ref{SUB},
\begin{eqnarray}\label{R3}
R_{3}^{\epsilon}(t)&=&\frac{4}{\epsilon^{2}}\mathbb{E}\Big\|\int_{s}^{t}\int_{0}^{u}y_{i}^{\epsilon}(u)y_{i}^{\epsilon}(r)^{-1}\nabla K*\bar{\rho}_{r}^{\epsilon}(\bar{x}_{i}^{\epsilon}(r))drdu\Big\|^{2}\nonumber\\
&\leq&\frac{4}{\epsilon^{2}}\mathbb{E}\Big(\int_{s}^{t}\int_{0}^{u}\|y_{i}^{\epsilon}(u)y_{i}^{\epsilon}(r)^{-1}\|\|\nabla K*\bar{\rho}_{r}^{\epsilon}(\bar{x}_{i}^{\epsilon}(r))\|drdu\Big)^{2}\nonumber\\
&\leq&\frac{C_{L_{K}}}{\epsilon^{2}}\mathbb{E}\Big(\int_{s}^{t}\int_{0}^{u}e^{-\frac{1}{\epsilon}(\lambda_{\gamma}+\lambda_{\phi})(u-r)}(1+\|\bar{x}_{i}^{\epsilon}(r)\|+\mathbb{E}\|\bar{x}_{i}^{\epsilon}(r)\|)drdu\Big)^{2}\nonumber\\
&\leq&\frac{C_{L_{K}}}{\epsilon^{2}}(1+\mathbb{E}\sup_{0\leq t\leq T}\|\bar{x}_{i}^{\epsilon}(t)\|^{2})\Big(\int_{s}^{t}\int_{0}^{u}e^{-\frac{2}{\epsilon}(\lambda_{\gamma}+\lambda_{\phi})(u-r)}drdu\Big)^{2}\nonumber\\
&\leq&C_{\lambda_{\gamma},\lambda_{\phi},L_{K},T}|t-s|(1+\mathbb{E}\|\bar{x}_{i,0}\|^{2}+\mathbb{E}\|\bar{v}_{i,0}\|^{2}).
\end{eqnarray}
For $R_{4}^{\epsilon}(t)$, by the H${\rm\ddot{o}}$lder inequality and the mean value theorem of integrals, there exists $\theta\in [s,t]$, such that $\int_{s}^{t}y_{i}^{\epsilon}(u)y_{i}^{\epsilon}(r)^{-1}du=y_{i}^{\epsilon}(\theta)y_{i}^{\epsilon}(r)^{-1}(t-s)$, we obtain
\begin{eqnarray}
&&R_{4}^{\epsilon}(t)=\frac{8}{\epsilon^{2}}\Big\|\int_{s}^{t}\int_{0}^{u}y_{i}^{\epsilon}(u)y_{i}^{\epsilon}(r)^{-1}\sigma(\bar{x}_{i}^{\epsilon}(r))dB_{i}(r)du\Big\|^{2}\nonumber\\
&&=\frac{8}{\epsilon^{2}}\mathbb{E}\Big\|\int_{0}^{t}\int_{0}^{u}y_{i}^{\epsilon}(u)y_{i}^{\epsilon}(r)^{-1}\sigma(\bar{x}_{i}^{\epsilon}(r))dB_{i}(r)du-\int_{0}^{s}\int_{0}^{u}y_{i}^{\epsilon}(u)y_{i}^{\epsilon}(r)^{-1}\sigma(\bar{x}_{i}^{\epsilon}(r))dB_{i}(r)du\Big\|^{2}\nonumber\\
&&=\frac{8}{\epsilon^{2}}\mathbb{E}\Big\|\int_{0}^{t}\int_{r}^{t}y_{i}^{\epsilon}(u)y_{i}^{\epsilon}(r)^{-1}\sigma(\bar{x}_{i}^{\epsilon}(r))dudB_{i}(r)-\int_{0}^{s}\int_{r}^{s}y_{i}^{\epsilon}(u)y_{i}^{\epsilon}(r)^{-1}\sigma(\bar{x}_{i}^{\epsilon}(r))dudB_{i}(r)\Big\|^{2}\nonumber
\end{eqnarray}
\begin{eqnarray}\label{R4}
&&\leq \frac{64}{\epsilon^{2}}\int_{0}^{s}\mathbb{E}\Big(\int_{s}^{t}\|\sigma(\bar{x}_{i}^{\epsilon}(r))\|e^{-\frac{1}{\epsilon}(\lambda_{\gamma}+\lambda_{\phi})(u-r)}du\Big)^{2}dr\nonumber\\
&&\quad+\frac{64}{\epsilon^{2}}\int_{s}^{t}\mathbb{E}\Big(\int_{r}^{t}\|\sigma(\bar{x}_{i}^{\epsilon}(r))\|e^{-\frac{1}{\epsilon}(\lambda_{\gamma}+\lambda_{\phi})(u-r)}du\Big)^{2}dr\nonumber\\
&&\leq\frac{64}{\epsilon^{2}}|t-s|^{2}\int_{0}^{s}\mathbb{E}\|\sigma(\bar{x}_{i}^{\epsilon}(r))\|^{2}e^{-\frac{2}{\epsilon}(C_{\lambda_{\gamma}}+C_{\lambda_{\phi}})(u-r)}dr\nonumber\\
&&\quad+\frac{64}{(\lambda_{\gamma}+\lambda_{\phi})^{2}}\int_{s}^{t}\mathbb{E}\|\sigma(\bar{x}_{i}^{\epsilon}(r))\|^{2}\Big(1-e^{-\frac{1}{\epsilon}(C_{\lambda_{\gamma}}+C_{\lambda_{\phi}})(t-r)}\Big)dr\nonumber\\
&&\leq \frac{64L_{\sigma}^{2}}{\epsilon^{2}}|t-s|^{2}\int_{0}^{s}(1+\mathbb{E}\sup_{0\leq u\leq r}\|\bar{x}_{i}^{\epsilon}(u)\|^{2})e^{-\frac{2}{\epsilon}(\lambda_{\gamma}+\lambda_{\phi})(\theta-r)}dr\nonumber\\
&&\quad+\frac{64L_{\sigma}^{2}}{(\lambda_{\gamma}+\lambda_{\phi})^{2}}\int_{s}^{t}(1+\mathbb{E}\sup_{0\leq u\leq r}\|\bar{x}_{i}^{\epsilon}(u)\|^{2})dr\nonumber\\
&&\leq C_{\lambda_{\gamma},\lambda_{\phi},L_{\sigma},T}|t-s|.
\end{eqnarray}
Then by (\ref{RZo}), (\ref{R1}), (\ref{R2R}),(\ref{R3}) and (\ref{R4}), we have
\begin{eqnarray*}
\mathbb{E}\|\bar{x}_{i}^{\epsilon}(t)-\bar{x}_{i}^{\epsilon}(s)\|^{2}\leq C_{\lambda_{\gamma},\lambda_{\phi},\sigma,L_{K},L_{V},T}(1+\mathbb{E}\|\bar{x}_{i,0}\|^{2}+\mathbb{E}\|\bar{v}_{i,0}\|^{2})|t-s|.
\end{eqnarray*}
\end{proof}

We further need some estimate on $\dot{\bar{x}}_{i}^{\epsilon}(t)$.
\begin{lemma}\label{EWJU}
Under assumptions {\rm($\mathbf{H}_{1}$)}-{\rm($\mathbf{H}_{4}$)}, for every $T>0$, 
\begin{eqnarray*}
\mathbb{E}\|\sqrt{\epsilon}\dot{\bar{x}}_{i}^{\epsilon}(t)\|^{2}\leq C_{T},
\end{eqnarray*}
where $C_{T}$ also depends on parameters $\lambda_{\gamma},\lambda_{\phi},L_{V},L_{K},L_{\sigma},\|K(0)\|,\bar{x}_{i,0},\bar{v}_{i,0}$.
\end{lemma}
\begin{proof}
By (\ref{EWJV}), $\dot{\bar{x}}_{i}^{\epsilon}(t)=\bar{v}_{i}^{\epsilon}(t)$ and It${\rm\hat{o}}$'s formula, 
\begin{eqnarray*}
\frac{1}{2}\frac{d}{dt}\|\sqrt{\epsilon}\bar{v}_{i}^{\epsilon}(t)\|^{2}&=&-\bar{v}_{i}^{\epsilon}(t)\cdot(\gamma(\bar{x}_{i}^{\epsilon})+\phi*\bar{\rho}_{t}^{\epsilon}(\bar{x}_{i}^{\epsilon}))\bar{v}_{i}^{\epsilon}(t)+\bar{v}_{i}^{\epsilon}(t)\cdot(\nabla V(\bar{x}_{i}^{\epsilon})+\nabla K*\bar{\rho}_{t}^{\epsilon}(\bar{x}_{i}^{\epsilon}))\\
&&+\bar{v}_{i}^{\epsilon}(t)\cdot\sigma(\bar{x}_{i}^{\epsilon}(t)\dot{B}_{i}(t))+\frac{1}{\epsilon}{\rm Tr}[\sigma(\bar{x}_{i}^{\epsilon}(t))\sigma(\bar{x}_{i}^{\epsilon}(t))^{*}].
\end{eqnarray*}
Furthermore, by Young's inequality,
\begin{eqnarray*}
\frac{1}{2}\frac{d}{dt}\mathbb{E}\|\sqrt{\epsilon}\bar{v}_{i}^{\epsilon}(t)\|^{2}&\leq&-\frac{\lambda_{\gamma}+\lambda_{\phi}}{2}\mathbb{E}\|\bar{v}_{i}^{\epsilon}(t)\|^{2}+\frac{1}{2(\lambda_{\gamma}+\lambda_{\phi})}\mathbb{E}\|\nabla V(\bar{x}_{i}^{\epsilon})+\nabla K*\bar{\rho}_{t}^{\epsilon}(\bar{x}_{i}^{\epsilon})\|^{2}\\
&&+\frac{1}{\epsilon}\mathbb{E}\|\sigma(\bar{x}_{i}^{\epsilon}(t))\|^{2}\\
&\leq&-\frac{\lambda_{\gamma}+\lambda_{\phi}}{2}\mathbb{E}\|\bar{v}_{i}^{\epsilon}(t)\|^{2}+\frac{L_{V}^{2}}{\lambda_{\gamma}+\lambda_{\phi}}\mathbb{E}\|\bar{x}_{i}^{\epsilon}(t)\|^{2}+\frac{L_{K}^{2}}{\lambda_{\gamma}+\lambda_{\phi}}\mathbb{E}^{i}[\mathbb{E}^{j}\|\bar{x}_{i}^{\epsilon}(t)-\bar{x}_{j}^{\epsilon}(t)\|^{2}]\\
&&+\frac{\|K(0)\|}{\lambda_{\gamma}+\lambda_{\phi}}+\frac{L_{\sigma}^{2}}{\epsilon}(1+\mathbb{E}\|\bar{x}_{i}^{\epsilon}(t)\|^{2})\\
&\leq&-\frac{\lambda_{\gamma}+\lambda_{\phi}}{2\epsilon}\mathbb{E}\|\sqrt{\epsilon}\bar{v}_{i}^{\epsilon}(t)\|^{2}+\frac{C_{\lambda_{\gamma},\lambda_{\phi},\sigma,L_{K},L_{V},L_{\sigma}}}{\epsilon}(1+\mathbb{E}\|\bar{x}_{i}^{\epsilon}(t)\|^{2}).
\end{eqnarray*}
Gronwall inequality yields
\begin{eqnarray*}
\mathbb{E}\|\sqrt{\epsilon}\bar{v}_{i}^{\epsilon}(t)\|^{2}\leq C_{\lambda_{\gamma},\lambda_{\phi},\sigma,L_{K},L_{V},L_{\sigma},T,\bar{x}_{i,0},\bar{v}_{i,0},\|K(0)\|}.
\end{eqnarray*}
\end{proof}

We also need some estimates to $\epsilon\mathbb{E}^{x}(\bar{v}_{i}^{\epsilon}(t)\otimes \bar{v}_{i}^{\epsilon}(t))$. By (\ref{EWJV}), for fixed $\bar{x}_{i}^{\epsilon}(t)=x, \bar{\rho}_{t}^{\epsilon}(x)=\bar{\rho}(x)$, let $\bar{v}_{i}^{\epsilon,x,\bar{\rho}}$ satisfy the following SDEs,
\begin{eqnarray}
\epsilon\dot{\bar{v}}_{i}^{\epsilon,x,\bar{\rho}}=-(\gamma(x)+\phi*\bar{\rho}(x))\bar{v}_{i}^{\epsilon,x,\bar{\rho}}-\nabla V(x)-\nabla K*\bar{\rho}(x)+\sigma(x)\dot{B}_{i}(t),
\end{eqnarray}
then for any $t\geq0$, $\mathbb{E}^{x}\|\bar{v}_{i}^{\epsilon,\bar{\rho}}(t)\|^{2}=\mathbb{E}\|\bar{v}_{i}^{\epsilon,x,\bar{\rho}}(t)\|^{2}$. We have the following result,
\begin{lemma}\label{lem:3.1}
Assuming that $(\mathbf{H_{1}})$-$(\mathbf{H_{4}})$ are valid, then for every $x\in\mathbb{R}^{d}$ and fixed $\bar{x}_{i}^{\epsilon}(t)=x, \bar{\rho}_{t}^{\epsilon}(t)=\bar{\rho}(x)$, we have 
\begin{eqnarray}
\mathbb{E}^{x}(\epsilon\bar{v}_{i}^{\epsilon,\bar{\rho}}(t)\otimes\bar{v}_{i}^{\epsilon,\bar{\rho}}(t))=J(x,\bar{\rho})+\epsilon C(x,\bar{\rho},t),
\end{eqnarray}
where $\|C(x,\bar{\rho},t)\|\leq C\Big(1+\|x\|^{2}+\int_{\mathbb{R}^{d}}\|y\|\bar{\rho}(y)dy\Big)$.
\end{lemma}
\begin{proof}
Lemma \ref{YT} yields
\begin{eqnarray}
\bar{v}_{i}^{\epsilon,x,\bar{\rho}}(t)&=&e^{-\frac{\gamma(x)+\phi*\bar{\rho}(x)}{\epsilon}t}\bar{v}_{i,0}\nonumber\\
&&-\frac{1}{\epsilon}\int_{0}^{t}e^{-\frac{\gamma(x)+\phi*\bar{\rho}(x)}{\epsilon}(t-s)}(\nabla V(x)+\nabla K*\bar{\rho}(x))ds\nonumber\\
&&+\frac{1}{\epsilon}\int_{0}^{t}e^{-\frac{\gamma(x)+\phi*\bar{\rho}(x)}{\epsilon}(t-s)}\sigma(x)dB_{i}(s),
\end{eqnarray}
then
\begin{eqnarray}\label{equ:3.12}
\mathbb{E}^{x}\bar{v}_{i}^{\epsilon,\bar{\rho}}(t)&=&e^{-\frac{\gamma(x)+\phi*\bar{\rho}(x)}{\epsilon}t}\mathbb{E}^{x}\bar{v}_{i,0}\nonumber\\
&-&(\gamma(x)+\phi*\bar{\rho}(x))^{-1}(I_{d\times d}-e^{-\frac{\gamma(x)+\phi*\bar{\rho}(x)}{\epsilon}t})\nonumber\\
&&\cdot(\nabla V(x)+\nabla K*\bar{\rho}(x)).
\end{eqnarray}
Applying It${\rm \hat{o}}$'s formula to $\bar{v}_{i}^{\epsilon,\bar{\rho}}(t)\otimes\bar{v}_{i}^{\epsilon,\bar{\rho}}(t)$,
\begin{eqnarray}\label{equ:3.13}
\frac{d}{dt}\mathbb{E}^{x}(\epsilon\bar{v}_{i}^{\epsilon,\bar{\rho}}(t)\otimes\bar{v}_{i}^{\epsilon,\bar{\rho}}(t))&=&-\frac{\gamma(x)+\phi*\bar{\rho}(x)}{\epsilon}\mathbb{E}^{x}(\epsilon\bar{v}_{i}^{\epsilon,\bar{\rho}}(t)\otimes\bar{v}_{i}^{\epsilon,\bar{\rho}}(t))\nonumber\\
&&-(\nabla V(x)+\nabla K*\bar{\rho}(x))\otimes \mathbb{E}^{x}\bar{v}_{i}^{\epsilon,\bar{\rho}}(t)\nonumber\\
&&-\mathbb{E}^{x}(\epsilon\bar{v}_{i}^{\epsilon,\bar{\rho}}(t)\otimes\bar{v}_{i}^{\epsilon,\bar{\rho}}(t))\frac{(\gamma(x)+\phi*\bar{\rho}(x))^{*}}{\epsilon}\nonumber\\
&&-\mathbb{E}^{x}\bar{v}_{i}^{\epsilon,\bar{\rho}}(t)\otimes(\nabla V(x)+\nabla K*\bar{\rho}(x))+\frac{1}{\epsilon}\sigma(x)\sigma(x)^{*}.
\end{eqnarray}\label{equ:3.14}
Then by Lemma \ref{YT},
\begin{eqnarray}\label{EVZ}
&&\mathbb{E}^{x}(\epsilon\bar{v}_{i}^{\epsilon,\bar{\rho}}(t)\otimes\bar{v}_{i}^{\epsilon,\bar{\rho}}(t))\nonumber\\
&=&e^{-\frac{\gamma(x)+\phi*\bar{\rho}(x)}{\epsilon}t}\mathbb{E}^{x}(\epsilon\bar{v}_{i,0}\otimes\bar{v}_{i,0})e^{-\frac{(\gamma(x)+\phi*\bar{\rho}(x))^{*}}{\epsilon}t}\nonumber\\
&&-\int_{0}^{t}e^{-\frac{\gamma(x)+\phi*\bar{\rho}(x)}{\epsilon}(t-s)}(\nabla V(x)+\nabla K*\bar{\rho}(x))\otimes\mathbb{E}^{x}\bar{v}_{i}^{\epsilon,\bar{\rho}}(s)e^{-\frac{(\gamma(x)+\phi*\bar{\rho}(x))^{*}}{\epsilon}(t-s)}ds\nonumber\\
&&-\int_{0}^{t}e^{-\frac{\gamma(x)+\phi*\bar{\rho}(x)}{\epsilon}(t-s)}\mathbb{E}^{x}\bar{v}_{i}^{\epsilon,\bar{\rho}}(s)\otimes(\nabla V(x)+\nabla K*\bar{\rho}(x))e^{-\frac{(\gamma(x)+\phi*\bar{\rho}(x))^{*}}{\epsilon}(t-s)}ds\nonumber\\
&&+\int_{0}^{t}e^{-\frac{\gamma(x)+\phi*\bar{\rho}(x)}{\epsilon}(t-s)}\sigma(x)\sigma(x)^{*}e^{-\frac{(\gamma(x)+\phi*\bar{\rho}(x))^{*}}{\epsilon}(t-s)}ds\nonumber\\
&\triangleq& \sum_{i=0}^{4}J_{i}^{\epsilon}(t).
\end{eqnarray}

Next we estimate the four terms respectively. In fact, by {\rm($\mathbf{H}_{3}$)} and {\rm($\mathbf{H}_{4}$)},
\begin{eqnarray}\label{equ:3.15}
\|J_{1}^{\epsilon}(t)\|&=&\|e^{-2\frac{\gamma(x)+\phi*\bar{\rho}(x)}{\epsilon}t}\|\|\mathbb{E}^{x}(\epsilon\bar{v}_{i,0}\otimes\bar{v}_{i,0})\|\|(e^{-2\frac{\gamma(x)+\phi*\bar{\rho}(x)}{\epsilon}t})^{*}\|\nonumber\\
&\leq&\epsilon e^{-2\frac{\lambda_{\gamma}+\lambda_{\phi}}{\epsilon}t}\mathbb{E}^{x}\|\bar{v}_{i,0}\|^{2}\to0,\quad \epsilon\to0.
\end{eqnarray}
By (\ref{equ:3.12}) and {\rm($\mathbf{H}_{1}$)}, {\rm($\mathbf{H}_{3}$)} and {\rm($\mathbf{H}_{4}$)}, for any $0\leq t\leq T$,
\begin{eqnarray}\label{EVH}
&&\|\mathbb{E}^{x}\bar{v}_{i}^{\epsilon,\bar{\rho}}(t)\|\nonumber\\
&\leq& \|e^{-\frac{\gamma(x)+\phi*\bar{\rho}(x)}{\epsilon}t}\mathbb{E}^{x}\bar{v}_{i,0}\|+\|(\gamma(x)+\phi*\bar{\rho}(x))^{-1}(\nabla V(x)+\nabla K*\bar{\rho}(x))\|\nonumber\\
&&+\|(\gamma(x)+\phi*\bar{\rho}(x))^{-1}e^{-\frac{\gamma(x)+\phi*\bar{\rho}(x)}{\epsilon}t}(\nabla V(x)+\nabla K*\bar{\rho}(x))\|\nonumber\\
&\leq&e^{-\frac{\lambda_{\gamma}+\lambda_{\phi}}{\epsilon}t}\mathbb{E}\|\bar{v}_{i,0}\|+\frac{1}{\lambda_{\gamma}+\lambda_{\phi}}\Big(L_{V}\|x\|+L_{K}\Big(1+\|x\|+\int_{\mathbb{R}^{d}}\|y\|\bar{\rho}(y)dy\Big)+\|\nabla K(0)\|\Big)\nonumber\\
&&+\frac{1}{\lambda_{\gamma}+\lambda_{\phi}}e^{-\frac{\lambda_{\gamma}+\lambda_{\phi}}{\epsilon}t}\Big(L_{V}\|x\|+L_{K}\Big(1+\|x\|+\int_{\mathbb{R}^{d}}\|y\|\bar{\rho}(y)dy\Big)+\|\nabla K(0)\|\Big)\nonumber\\
&\leq&C_{\lambda_{\gamma},\lambda_{\phi},L_{K},L_{V},\bar{v}_{i,0},\|\nabla K(0)\|}\Big(1+\|x\|+\int_{\mathbb{R}^{d}}\|y\|\bar{\rho}(y)dy\Big).
\end{eqnarray}
Then for $J_{2}^{\epsilon}(t)$, by (\ref{EVH}), {\rm($\mathbf{H}_{1}$)}, {\rm($\mathbf{H}_{3}$)} and {\rm($\mathbf{H}_{4}$)},
\begin{eqnarray}\label{equ:3.16}
\|J_{2}^{\epsilon}(t)\|
&\leq&\int_{0}^{t}e^{-2\frac{\lambda_{\gamma}+\lambda_{\phi}}{\epsilon}(t-s)}\|\nabla V(x)+\nabla K*\bar{\rho}(x)\|\|\mathbb{E}^{x}\bar{v}_{i}^{\epsilon,\bar{\rho}}(s)\|ds\nonumber\\
&\leq&\epsilon C_{\lambda_{\gamma},\lambda_{\phi},L_{K},L_{V},\bar{v}_{i,0},\|\nabla K(0)\|}\Big(1+\|x\|+\int_{\mathbb{R}^{d}}\|y\|\bar{\rho}(y)dy\Big).
\end{eqnarray}
Similariy,
\begin{eqnarray}\label{equ:3.17}
\|J_{3}^{\epsilon}(t)\|
&\leq&\epsilon C_{\lambda_{\gamma},\lambda_{\phi},L_{K},L_{V},\bar{v}_{i,0},\|\nabla K(0)\|}\Big(1+\|x\|+\int_{\mathbb{R}^{d}}\|y\|\bar{\rho}(y)dy\Big).
\end{eqnarray}
Further, for $J_{4}^{\epsilon}(t)$, let $\tau=\frac{t-s}{\epsilon}$,
\begin{eqnarray}\label{equ:3.18}    
J_{4}^{\epsilon}(t)&=&\int_{0}^{\frac{t}{\epsilon}}e^{-(\gamma(x)+\phi*\bar{\rho}(x))\tau}\sigma(x)\sigma(x)^{*}e^{-(\gamma(x)+\phi*\bar{\rho}(x))^{*}\tau}d\tau\nonumber\\
&=&\int_{0}^{\infty}e^{-(\gamma(x)+\phi*\bar{\rho}(x))\tau}\sigma(x)\sigma(x)^{*}e^{-(\gamma(x)+\phi*\bar{\rho}(x))^{*}\tau}d\tau\nonumber\\
&&-\int_{\frac{t}{\epsilon}}^{\infty}e^{-(\gamma(x)+\phi*\bar{\rho}(x))\tau}\sigma(x)\sigma(x)^{*}e^{-(\gamma(x)+\phi*\bar{\rho}(x))^{*}\tau}d\tau\nonumber\\
&\triangleq&J_{4,1}^{\epsilon}(t)+J_{4,2}^{\epsilon}(t).
\end{eqnarray}
Note that 
\begin{eqnarray*}
\|J_{4,2}^{\epsilon}(t)\|&\leq&\int_{\frac{t}{\epsilon}}^{\infty}e^{-2(\lambda_{\gamma}+\lambda_{\phi})\tau}\|\sigma(x)\|^{2}d\tau\\
&\leq&C_{L_{\sigma}}(1+\|x\|^{2})\int_{\frac{t}{\epsilon}}^{\infty}e^{-2(\lambda_{\gamma}+\lambda_{\phi})\tau}d\tau\leq\epsilon C_{L_{\sigma},\lambda_{\gamma},\lambda_{\phi}}(1+\|x\|^{2}).
\end{eqnarray*}
Thus, by (\ref{EVZ}), (\ref{equ:3.15}), (\ref{equ:3.16}), (\ref{equ:3.17}), (\ref{equ:3.18}), we have
\begin{eqnarray*}
\mathbb{E}^{x}(\epsilon\bar{v}_{i}^{\epsilon,\bar{\rho}}(t)\otimes\bar{v}_{i}^{\epsilon,\bar{\rho}}(t))=J(x,\bar{\rho})+\epsilon C(x,\bar{\rho},t),
\end{eqnarray*}
where $\|C(x,\bar{\rho},t)\|\leq C_{\lambda_{\gamma},\lambda_{\phi},L_{K},L_{V},\bar{v}_{i,0},\|\nabla K(0)\|}\Big(1+\|x\|^{2}+\int_{\mathbb{R}^{d}}\|y\|\bar{\rho}(y)dy\Big)$ and 
\begin{eqnarray}\label{JBU}
J(x,\bar{\rho})=\int_{0}^{\infty}e^{-(\gamma(x)+\phi*\bar{\rho}(x))t}\sigma(x)\sigma(x)^{*}e^{-(\gamma(x)+\phi*\bar{\rho}(x))^{*}t}dt,
\end{eqnarray}
here $\bar{\rho}$ is an arbitrary Gaussian distribution.
\end{proof}
\begin{remark}
By $(\mathbf{H_{2}})$, $(\mathbf{H_{3}})$ and $(\mathbf{H_{4}})$, 
\begin{eqnarray}\label{EWJ}
\|J(x,\bar{\rho})\|\leq \int_{0}^{\infty}e^{-2(\lambda_{\gamma}+\lambda_{\phi})t}\|\sigma(x)\|^{2}dt\leq C_{\lambda_{\gamma},\lambda_{\phi},L_{\sigma}}(1+\|x\|^{2}).
\end{eqnarray}
\end{remark}
In order to estimate the difference between the process $Y_{t}^{\epsilon}$ (See (\ref{equ:3.3}) below) and the auxiliary processes $\hat{Y}_{t}^{\epsilon}$ (See (\ref{equ:3.5}) below), we give the following result.
\begin{lemma}\label{YPLR}
Under assumptions $(\mathbf{H_{1}})$-$(\mathbf{H_{4}})$, for every $T>0$ and $\psi\in C_{0}^{\infty}(\mathbb{R}^{d};\mathbb{R}^{d})$, 
\begin{eqnarray*}\label{Yz}
\sup_{0\leq t\leq T}|\langle Y_{t}^{\epsilon},\psi\rangle|&\leq& C_{T,\bar{x}_{i,0},\bar{v}_{i,0},L_{V},L_{K},\lambda_{\gamma},\lambda_{\phi}}\|\psi\|_{Lip}.
\end{eqnarray*}
Here $\langle\cdot,\cdot\rangle$ denotes the inner product on $L^{2}(\mathbb{R}^{d})$ and $\|\cdot\|_{Lip}$ denotes the Lipschitz norm defined by 
\begin{eqnarray*}
\|f\|_{Lip}=\|f\|_{\infty}+\sup_{x\neq y}\frac{\|f(x)-f(y)\|}{\|x-y\|}.
\end{eqnarray*}
\end{lemma}
\begin{proof}
By the chain rule, (\ref{equ:3.4.2}) and $(\mathbf{H_{1}})$-$(\mathbf{H_{4}})$,
\begin{eqnarray*}\label{equ:3.19}
&&\frac{1}{2}\frac{d}{dt}|\langle Y_{t}^{\epsilon},\psi\rangle|^{2}=\langle Y_{t}^{\epsilon},\psi\rangle\Big[-\Big\langle \frac{\gamma(x)+\phi*\bar{\rho}_{t}^{\epsilon}(x)}{\epsilon}Y_{t}^{\epsilon}, \psi\Big\rangle-\Big\langle\frac{\nabla V(x)+\nabla K*\bar{\rho}_{t}^{\epsilon}(x)}{\epsilon}\bar{\rho}_{t}^{\epsilon},\psi\Big\rangle\nonumber\\
&&\quad\quad\quad\quad+\langle\bar{\rho}_{t}^{\epsilon}\mathbb{E}^{x}(\bar{v}_{i}^{\epsilon}(t)\otimes \bar{v}_{i}^{\epsilon}(t)),\nabla_{x}\psi\rangle\Big]\nonumber\\
&&\leq-\frac{\lambda_{\gamma}+\lambda_{\phi}}{\epsilon}|\langle Y_{t}^{\epsilon},\psi\rangle|^{2}+\frac{1}{\epsilon}|\langle Y_{t}^{\epsilon},\psi\rangle||\mathbb{E}[(\nabla V(\bar{x}_{i}^{\epsilon})+\mathbb{E}^{j}\nabla K(\bar{x}_{i}^{\epsilon}(t)-\bar{x}_{j}^{\epsilon}(t)))\psi(\bar{x}_{i}^{\epsilon}(t))]|\nonumber\\
&&\quad\quad\quad\quad+\frac{1}{\epsilon}|\langle Y_{t}^{\epsilon},\psi\rangle||\mathbb{E}[\mathbb{E}^{\bar{x}_{i}^{\epsilon}(t)}(\epsilon\bar{v}_{i}^{\epsilon}(t)\otimes \bar{v}_{i}^{\epsilon}(t))\nabla_{x}\psi(\bar{x}_{i}^{\epsilon}(t))]|\nonumber\\
&&\leq-\frac{\lambda_{\gamma}+\lambda_{\phi}}{2\epsilon}|\langle Y_{t}^{\epsilon},\psi\rangle|^{2}+\frac{1}{(\lambda_{\gamma}+\lambda_{\phi})\epsilon}|\mathbb{E}[(\nabla V(\bar{x}_{i}^{\epsilon})+\mathbb{E}^{j}\nabla K(\bar{x}_{i}^{\epsilon}(t)-\bar{x}_{j}^{\epsilon}(t)))\psi(\bar{x}_{i}^{\epsilon}(t))]|^{2}\nonumber\\
&&\quad\quad\quad\quad+\frac{1}{(\lambda_{\gamma}+\lambda_{\phi})\epsilon}|\mathbb{E}[\mathbb{E}^{\bar{x}_{i}^{\epsilon}(t)}(\epsilon\bar{v}_{i}^{\epsilon}(t)\otimes \bar{v}_{i}^{\epsilon}(t))\nabla_{x}\psi(\bar{x}_{i}^{\epsilon}(t))]|^{2}\nonumber\\
&&\leq-\frac{\lambda_{\gamma}+\lambda_{\phi}}{2\epsilon}|\langle Y_{t}^{\epsilon},\psi\rangle|^{2}+\frac{\|\psi\|_{Lip}^{2}}{(\lambda_{\gamma}+\lambda_{\phi})\epsilon}\mathbb{E}[L_{V}\|\bar{x}_{i}^{\epsilon}(t)\|+L_{K}\mathbb{E}^{j}(1+\|\bar{x}_{i}^{\epsilon}(t)\|+\|\bar{x}_{j}^{\epsilon}(t)\|)+\|\nabla K(0)\|]^{2}\nonumber\\
&&\quad+\frac{1}{(\lambda_{\gamma}+\lambda_{\phi})\epsilon}\|\psi\|_{Lip}^{2}\|\mathbb{E}[\|J(x,\bar{\rho}_{t}^{\epsilon})\|+\epsilon \|C(x,\bar{\rho}_{t}^{\epsilon},t)\|]^{2}\nonumber\\
&&\leq-\frac{\lambda_{\gamma}+\lambda_{\phi}}{2\epsilon}|\langle Y_{t}^{\epsilon},\psi\rangle|^{2}+\frac{C_{L_{V},L_{K},\lambda_{\gamma},\lambda_{\phi}}}{\epsilon}(1+\mathbb{E}\|\bar{x}_{i}^{\epsilon}(t)\|^{2})\|\psi\|_{Lip}^{2}
\end{eqnarray*}
Gronwall's inequality and Lemma \ref{SUB} yield
\begin{eqnarray*}\label{Yz}
\sup_{0\leq t\leq T}|\langle Y_{t}^{\epsilon},\psi\rangle|&\leq& C_{T,\bar{x}_{i,0},\bar{v}_{i,0},L_{V},L_{K},\lambda_{\gamma},\lambda_{\phi}}\|\psi\|_{Lip}.
\end{eqnarray*}
\end{proof}

By Lemma \ref{SUB} and Lemma \ref{SC}, we obtain that $\{\bar{x}_{i}^{\epsilon}(t)\}_{0< \epsilon\leq1}$ is tight in space $C(0,T;\mathbb{R}^{d})$. Next we pass the limit in (\ref{equ:main1}) as $\epsilon\to0$.

Let $f_{t}^{\epsilon}(x,v)$ be the law of $(\bar{x}_{i}^{\epsilon}(t),\bar{v}_{i}^{\epsilon}(t))$, then rewriting the (\ref{equ:main1}),
\begin{eqnarray*}
&&\dot{\bar{x}}_{i}^{\epsilon}=\bar{v}_{i}^{\epsilon}\\
&&\epsilon\dot{\bar{v}}_{i}^{\epsilon}=-(\gamma(\bar{x}_{i}^{\epsilon})+\phi*\bar{\rho}_{t}^{\epsilon}(\bar{x}_{i}^{\epsilon}))\bar{v}_{i}^{\epsilon}-\nabla V(\bar{x}_{i}^{\epsilon})-\nabla K*\bar{\rho}_{t}^{\epsilon}(\bar{x}_{i}^{\epsilon})+\sigma(\bar{x}_{i}^{\epsilon})\dot{B}_{i}(t),
\end{eqnarray*}
then $f_{t}^{\epsilon}(x,v)$ satisfies the following Vlasov-Fokker-Planck equation
\begin{eqnarray}\label{equ:3.101}
\partial_{t}f_{t}^{\epsilon}+v\nabla_{x}f_{t}^{\epsilon}&=&\frac{\gamma(x)+\phi*\bar{\rho}_{t}^{\epsilon}}{\epsilon}\nabla_{v}(vf_{t}^{\epsilon})-\Big(\frac{\nabla V(x)}{\epsilon}+\frac{\nabla K*\bar{\rho}_{t}^{\epsilon}}{\epsilon}\Big)\nabla_{v}f_{t}^{\epsilon}\nonumber\\
&&+\frac{\sigma(x)}{\epsilon^{2}}\Delta_{v}f_{t}^{\epsilon},
\end{eqnarray}
in the weak sense, that is, for $\Psi\in C_{0}^{\infty}(\mathbb{R}^{d}\times\mathbb{R}^{d})$,
\begin{eqnarray}\label{equ:3.2}
&&\langle f_{t}^{\epsilon},\Psi\rangle-\langle f_{0}^{\epsilon},\Psi\rangle\nonumber\\
&=&\int_{0}^{t}\int_{\mathbb{R}^{d}\times\mathbb{R}^{d}}\Big[v\nabla_{x}\Psi-\Big(\frac{\gamma(x)+\phi*\bar{\rho}_{s}^{\epsilon}(x)}{\epsilon}v-\frac{\nabla V(x)+\nabla K*\bar{\rho}_{s}^{\epsilon}(x)}{\epsilon}\Big)\nabla_{v}\Psi\nonumber\\
&&+\frac{\sigma}{\epsilon^{2}}\Delta_{v}\Psi\Big]f_{s}^{\epsilon}(x,v)dxdv.
\end{eqnarray}

For small $\epsilon>0$, the above Vlasov-Fokker-Planck equation (\ref{equ:3.101}) is a singular system which is difficult to pass the limit $\epsilon\to0$. Here we present an averaging approach to derive such limit for $f_{t}^{\epsilon}$ explicitly.

For this, we define the local mass $\bar{\rho}_{t}^{\epsilon}$, the marginal distribution of $f_{t}^{\epsilon}$, and local momentum $Y_{t}^{\epsilon}$ as
\begin{eqnarray*}
\bar{\rho}_{t}^{\epsilon}(x)=\int_{\mathbb{R}^{d}}f_{t}^{\epsilon}(x,v)dv, \quad Y_{t}^{\epsilon}=\int_{\mathbb{R}^{d}}vf_{t}^{\epsilon}(x,v)dv,
\end{eqnarray*}
respectively. Then $\bar{\rho}_{t}^{\epsilon}$ is the law of $\bar{x}_{i}^{\epsilon}(t)$ and in the weak sense,
\begin{eqnarray}\label{equ:3.3}
\partial_{t}\bar{\rho}_{t}^{\epsilon}&=&-\nabla_{x} Y_{t}^{\epsilon}\\
\partial_{t}Y_{t}^{\epsilon}&=&-\frac{\gamma(x)+\phi*\bar{\rho}_{t}^{\epsilon}}{\epsilon}Y_{t}^{\epsilon}-\frac{\nabla V(x)+\nabla K*\bar{\rho}_{t}^{\epsilon}(x)}{\epsilon}\bar{\rho}_{t}^{\epsilon}\nonumber\\
&&-\nabla_{x}\cdot[\bar{\rho}_{t}^{\epsilon}\mathbb{E}^{x}(\bar{v}_{i}^{\epsilon}(t)\otimes \bar{v}_{i}^{\epsilon}(t))].
\end{eqnarray}
Here we have used the following calculations
\begin{equation*}
\int_{\mathbb{R}^{d}}v\otimes vf_{t}^{\epsilon}(x,v)dv=\bar{\rho}_{t}^{\epsilon}(x)\int_{\mathbb{R}^{d}}v\otimes v\frac{f_{t}^{\epsilon}(x,v)}{\bar{\rho}_{t}^{\epsilon}(x)}dv=\bar{\rho}_{t}^{\epsilon}(x)\mathbb{E}^{x}(\bar{v}_{i}^{\epsilon}(t)\otimes \bar{v}_{i}^{\epsilon}(t)).
\end{equation*}
We couple equation (\ref{equ:main1}) to the above slow-fast system which yields the following closed system
\begin{eqnarray}\label{equ:3.4}
\partial_{t}\bar{\rho}_{t}^{\epsilon}&=&-\nabla_{x}Y_{t}^{\epsilon},\label{Macro1}\\
\partial_{t}Y_{t}^{\epsilon}&=&-\frac{\gamma(x)+\phi*\bar{\rho}_{t}^{\epsilon}}{\epsilon}Y_{t}^{\epsilon}-\frac{\nabla V(x)+\nabla K*\bar{\rho}_{t}^{\epsilon}(x)}{\epsilon}\bar{\rho}_{t}^{\epsilon}\label{equ:3.4.2}\\
&&-\nabla_{x}\cdot[\bar{\rho}_{t}^{\epsilon}\mathbb{E}^{x}(\bar{v}_{i}^{\epsilon}(t)\otimes \bar{v}_{i}^{\epsilon}(t))],\nonumber\\
\dot{\bar{x}}_{i}^{\epsilon}&=&\bar{v}_{i}^{\epsilon},\\
\epsilon\dot{\bar{v}}_{i}^{\epsilon}&=&-(\gamma(\bar{x}_{i}^{\epsilon})+\phi*\bar{\rho}_{t}^{\epsilon}(\bar{x}_{i}^{\epsilon}))\bar{v}_{i}^{\epsilon}-\nabla V(\bar{x}_{i}^{\epsilon})-\nabla K*\bar{\rho}_{t}^{\epsilon}(\bar{x}_{i}^{\epsilon})\label{equ:3.4.3}\\
&&+\sigma(\bar{x}_{i}^{\epsilon})\dot{B}_{i}(t).\nonumber
\end{eqnarray}
Here $\mathbb{E}^{x}$ is the expectation with fixed $\bar{x}_{i}^{\epsilon}(t)=x\in\mathbb{R}^{d}$. We consider the equations (\ref{Macro1})-(\ref{equ:3.4.2}) in the weak sense as we just concerned with the weak convergence of $\bar{\rho}_{t}^{\epsilon}$.
Next we apply an averaging approach to pass the limit $\epsilon\to0$. We need an auxiliary process $\{\hat{\rho}_{t}^{\epsilon}, \hat{Y}_{t}^{\epsilon}\}_{0\leq t\leq T}$. For this we divide the time interval $[0,T]$ into intervals of size $\delta>0$ as $0=t_{0}<t_{1}<\cdots<t_{k}<t_{k+1}<\cdots<t_{[\frac{T}{\delta}]+1}=T$, for $t_{k+1}-t_{k}=\delta, k=0,\cdots,[\frac{T}{\delta}]$ and for $t\in[t_{k},t_{k+1}]$,
\begin{eqnarray}\label{equ:3.5}
\partial_{t}\hat{\rho}_{t}^{\epsilon}&=&-\nabla_{x}\hat{Y}_{t}^{\epsilon},\quad \hat{\rho}_{t_{k}}^{\epsilon}=\bar{\rho}_{t_{k}}^{\epsilon},\\
\partial_{t}\hat{Y}_{t}^{\epsilon}&=&-\frac{\gamma(x)+\phi*\bar{\rho}_{t_{k}}^{\epsilon}(x)}{\epsilon}\hat{Y}_{t}^{\epsilon}-\frac{\nabla V(x)+\nabla K*\bar{\rho}_{t_{k}}^{\epsilon}(x)}{\epsilon}\bar{\rho}_{t_{k}}^{\epsilon}\nonumber\\
&&-\nabla_{x}\cdot[\bar{\rho}_{t_{k}}^{\epsilon}\mathbb{E}^{x}(\bar{v}_{i}^{\epsilon}(t_{k})\otimes \bar{v}_{i}^{\epsilon}(t_{k})],\\
\hat{Y}_{t_{k}}^{\epsilon}&=&Y_{t_{k}}^{\epsilon},\nonumber
\end{eqnarray}
Next lemma gives an estimate of the difference between $\hat{Y}_{t}^{\epsilon}$ and $Y_{t}^{\epsilon}$.
\begin{lemma}\label{Le3}
Assuming that $(\mathbf{H_{1}})$-$(\mathbf{H_{4}})$ are valid, for every $T>0$ and $\psi\in C_{0}^{\infty}(\mathbb{R}^{d};\mathbb{R}^{d})$, there is a constant $C_{T}>0$ such that
\begin{eqnarray*}
\sup_{0\leq t\leq T}|\langle Y_{t}^{\epsilon}-\hat{Y}_{t}^{\epsilon},\psi\rangle|\leq C_{T,\bar{x}_{i,0},\bar{v}_{i,0},\lambda_{\gamma},\lambda_{\phi},L_{\sigma},L_{V},L_{K}}\Big(\frac{\delta}{\epsilon}+\frac{\delta}{\epsilon^{2}}\Big)e^{C_{L_{\phi}}\frac{\delta}{\epsilon}}\|\phi\|_{Lip}.
\end{eqnarray*}
\end{lemma}

\begin{proof}
Let $Z_{t}^{\epsilon}=Y_{t}^{\epsilon}-\hat{Y}_{t}^{\epsilon}$, then for $t\in[t_{k},t_{k+1}]$, 
\begin{eqnarray*}
\partial_{t}Z_{t}^{\epsilon}&=&-\frac{\gamma(x)+\phi*\bar{\rho}_{t}^{\epsilon}(x)}{\epsilon}Y_{t}^{\epsilon}
+\frac{\gamma(x)+\phi*\bar{\rho}_{t_{k}}^{\epsilon}(x)}{\epsilon}\hat{Y}_{t}^{\epsilon}\\
&&-\frac{\nabla V(x)+\nabla K*\bar{\rho}_{t}^{\epsilon}(x)}{\epsilon}\bar{\rho}_{t}^{\epsilon}
+\frac{\nabla V(x)+\nabla K*\bar{\rho}_{t_{k}}^{\epsilon}(x)}{\epsilon}\bar{\rho}_{t_{k}}^{\epsilon}\\
&&-\nabla_{x}\cdot[\bar{\rho}_{t}^{\epsilon}\mathbb{E}^{x}(\bar{v}_{i}^{\epsilon}(t)\otimes \bar{v}_{i}^{\epsilon}(t)]
+\nabla_{x}\cdot[\bar{\rho}_{t_{k}}^{\epsilon}\mathbb{E}^{x}(\bar{v}_{i}^{\epsilon}(t_{k})\otimes \bar{v}_{i}^{\epsilon}(t_{k})],
\end{eqnarray*}
and by Lemma \ref{YT}, 
\begin{eqnarray}\label{equ:3.6}
&&Z_{t}^{\epsilon}=-\frac{1}{\epsilon}\int_{t_{k}}^{t}e^{-\frac{\gamma(x)}{\epsilon}(t-s)}(\phi*\bar{\rho}_{t}^{\epsilon}(x)Y_{s}^{\epsilon}-\phi*\bar{\rho}_{t_{k}}^{\epsilon}(x)\hat{Y}_{s}^{\epsilon})ds\\
&&\quad-\frac{1}{\epsilon}\int_{t_{k}}^{t}e^{-\frac{\gamma(x)}{\epsilon}(t-s)}[(\nabla V(x)+\nabla K*\bar{\rho}_{s}^{\epsilon}(x))\bar{\rho}_{s}^{\epsilon}\nonumber\\
&&\quad\quad\quad\quad\quad\quad\quad\quad\quad\quad\quad-(\nabla V(x)+\nabla K*\bar{\rho}_{t_{k}}^{\epsilon}(x))\bar{\rho}_{t_{k}}^{\epsilon}]ds\nonumber\\
&&-\int_{t_{k}}^{t}e^{-\frac{\gamma(x)}{\epsilon}(t-s)}\nabla_{x}\cdot[\bar{\rho}_{s}^{\epsilon}\mathbb{E}^{x}(\bar{v}_{i}^{\epsilon}(s)\otimes \bar{v}_{i}^{\epsilon}(s))-\bar{\rho}_{t_{k}}^{\epsilon}\mathbb{E}^{x}(\bar{v}_{i}^{\epsilon}(t_{k})\otimes \bar{v}_{i}^{\epsilon}(t_{k}))]ds.\nonumber
\end{eqnarray}
Then for any $\psi\in C_{0}^{\infty}(\mathbb{R}^{d};\mathbb{R}^{d})$,
\begin{eqnarray}
\langle Z_{t}^{\epsilon},\psi\rangle&=&-\frac{1}{\epsilon}\int_{t_{k}}^{t}\langle e^{-\frac{\gamma(x)}{\epsilon}(t-s)}(\phi*\bar{\rho}_{s}^{\epsilon}(x)Y_{s}^{\epsilon}-\phi*\bar{\rho}_{t_{k}}^{\epsilon}(x)\hat{Y}_{s}^{\epsilon}),\psi\rangle ds\nonumber
\end{eqnarray}
\begin{eqnarray}\label{Z}
&&-\frac{1}{\epsilon}\int_{t_{k}}^{t}\big\langle e^{-\frac{\gamma(x)}{\epsilon}(t-s)}[(\nabla V(x)+\nabla K*\bar{\rho}_{s}^{\epsilon}(x))\bar{\rho}_{s}^{\epsilon}\nonumber\\
&&\quad\quad\quad\quad\quad\quad\quad\quad-(\nabla V(x)+\nabla K*\bar{\rho}_{t_{k}}^{\epsilon}(x))\bar{\rho}_{t_{k}}^{\epsilon}],\psi\big\rangle ds\nonumber\\
&&-\int_{t_{k}}^{t}\big\langle e^{-\frac{\gamma(x)}{\epsilon}(t-s)}\nabla_{x}\cdot[\bar{\rho}_{s}^{\epsilon}\mathbb{E}^{x}(\bar{v}_{i}^{\epsilon}(s)\otimes \bar{v}_{i}^{\epsilon}(s))\nonumber\\
&&\quad\quad\quad\quad\quad\quad\quad\quad-\bar{\rho}_{t_{k}}^{\epsilon}\mathbb{E}^{x}(\bar{v}_{i}^{\epsilon}(t_{k})\otimes \bar{v}_{i}^{\epsilon}(t_{k}))],\psi\big\rangle ds\nonumber\\
&\triangleq&\Pi_{1}^{\epsilon}(t)+\Pi_{2}^{\epsilon}(t)+\Pi_{3}^{\epsilon}(t).
\end{eqnarray}
We calculate the three terms respectively. In fact, for $\Pi_{1}^{\epsilon}(t)$, 
by $(\mathbf{H_{3}})$, $(\mathbf{H_{4}})$ and Lemma \ref{YPLR}, 
\begin{eqnarray}\label{equ:3.8}
\Pi_{1}^{\epsilon}(t)&\leq&\frac{1}{\epsilon}\int_{t_{k}}^{t}|\langle e^{-\frac{\gamma(x)}{\epsilon}(t-s)}[\mathbb{E}\phi(x-\bar{x}_{i}^{\epsilon}(s))Y_{s}^{\epsilon}-\mathbb{E}\phi(x-\bar{x}_{i}^{\epsilon}(t_{k}))\hat{Y}_{s}^{\epsilon}],\psi\rangle|ds\nonumber\\
&&+\frac{1}{\epsilon}\int_{t_{k}}^{t}|\langle e^{-\frac{\gamma(x)}{\epsilon}(t-s)}\mathbb{E}\phi(x-\bar{x}_{i}^{\epsilon}(t_{k}))Z_{s}^{\epsilon},\psi\rangle|ds\nonumber\\
&=&\frac{1}{\epsilon}\int_{t_{k}}^{t}\Big|\int_{\mathbb{R}^{d}}e^{-\frac{\gamma(x)}{\epsilon}(t-s)}[\mathbb{E}\phi(x-\bar{x}_{i}^{\epsilon}(s))-\mathbb{E}\phi(x-\bar{x}_{i}^{\epsilon}(t_{k}))]Y_{s}^{\epsilon}(x)\psi(x)dx\Big|ds\nonumber\\
&&+\frac{1}{\epsilon}\int_{t_{k}}^{t}\Big|\int_{\mathbb{R}^{d}}e^{-\frac{\gamma(x)}{\epsilon}(t-s)}\mathbb{E}\phi(x-\bar{x}_{i}^{\epsilon}(t_{k}))Z_{s}^{\epsilon}(x)\psi(x)dx\Big|ds\nonumber\\
&\leq&\frac{2L_{\phi}}{\epsilon}\int_{t_{k}}^{t}e^{-\frac{\lambda_{\gamma}}{\epsilon}(t-s)}|\langle Y_{s}^{\epsilon},\psi\rangle|ds+\frac{L_{\phi}}{\epsilon}\int_{t_{k}}^{t}e^{-\frac{\lambda_{\gamma}}{\epsilon}(t-s)}|\langle Z_{s}^{\epsilon},\psi\rangle|ds\nonumber\\
&\leq&\frac{\delta}{\epsilon}C_{T,\bar{x}_{i,0},\bar{v}_{i,0},L_{V},L_{K},L_{\phi},\lambda_{\gamma},\lambda_{\phi}}\|\psi\|_{Lip}+\frac{L_{\phi}}{\epsilon}\int_{t_{k}}^{t}\sup_{0\leq u\leq s}|\langle Z_{u}^{\epsilon},\psi\rangle|ds.
\end{eqnarray}
Similarly, by $(\mathbf{H_{1}})$-$(\mathbf{H_{4}})$ together with the proof of Lemma \ref{YPLR},
\begin{eqnarray}\label{equ:3.9}
|\Pi_{2}^{\epsilon}(t)|&\leq&\Big|\frac{1}{\epsilon}\int_{t_{k}}^{t}\langle e^{-\frac{\gamma(x)}{\epsilon}(t-s)}(\nabla V(x)+\nabla K*\bar{\rho}_{s}^{\epsilon}(x)))\bar{\rho}_{s}^{\epsilon},\psi\rangle ds\Big|\nonumber\\
&&+\Big|\frac{1}{\epsilon}\int_{t_{k}}^{t}\langle e^{-\frac{\gamma(x)}{\epsilon}(t-s)}(\nabla V(x)+\nabla K*\bar{\rho}_{t_{k}}^{\epsilon}(x))\bar{\rho}_{t_{k}}^{\epsilon},\psi\rangle ds\Big|\nonumber\\
&=&\frac{1}{\epsilon}\Big|\int_{t_{k}}^{t}\mathbb{E}\Big[e^{-\frac{\gamma(\bar{x}_{i}^{\epsilon}(s))}{\epsilon}(t-s)}(\nabla V(\bar{x}_{i}^{\epsilon}(s))+\mathbb{E}^{j}\nabla K(\bar{x}_{i}^{\epsilon}(s)-\bar{x}_{j}^{\epsilon}(s)))\psi(\bar{x}_{i}^{\epsilon}(s))\Big]ds\Big|\nonumber\\
&&+\frac{1}{\epsilon}\Big|\int_{t_{k}}^{t}\mathbb{E}\Big[e^{-\frac{\gamma(\bar{x}_{i}^{\epsilon}(t_{k}))}{\epsilon}(t-s)}(\nabla V(\bar{x}_{i}^{\epsilon}(t_{k}))+\mathbb{E}^{j}\nabla K(\bar{x}_{i}^{\epsilon}(t_{k})-\bar{x}_{j}^{\epsilon}(t_{k})))\psi(\bar{x}_{i}^{\epsilon}(t_{k}))\Big]ds\Big|\nonumber\\
&\leq&\frac{\delta}{\epsilon}C_{T,\bar{x}_{i,0},\bar{v}_{i,0},L_{V},L_{K}}\|\psi\|_{Lip}.
\end{eqnarray}
For $\Pi_{3}^{\epsilon}(t)$, by Lemma \ref{YL},
\begin{eqnarray*}
\Pi_{3}^{\epsilon}(t)&=&-\int_{t_{k}}^{t}\langle e^{-\frac{\gamma(x)}{\epsilon}(t-s)}\nabla_{x}\cdot[\bar{\rho}_{s}^{\epsilon}\mathbb{E}^{x}(\bar{v}_{i}^{\epsilon}(s)\otimes \bar{v}_{i}^{\epsilon}(s))-\bar{\rho}_{t_{k}}^{\epsilon}\mathbb{E}^{x}(\bar{v}_{i}^{\epsilon}(t_{k})\otimes \bar{v}_{i}^{\epsilon}(t_{k}))],\psi\rangle ds\nonumber
\end{eqnarray*}
\begin{eqnarray*}
&=&-\int_{t_{k}}^{t}\langle\nabla_{x}\cdot[\bar{\rho}_{s}^{\epsilon}\mathbb{E}^{x}(\bar{v}_{i}^{\epsilon}(s)\otimes \bar{v}_{i}^{\epsilon}(s))-\bar{\rho}_{t_{k}}^{\epsilon}\mathbb{E}^{x}(\bar{v}_{i}^{\epsilon}(t_{k})\otimes \bar{v}_{i}^{\epsilon}(t_{k}))],e^{-\frac{\gamma(x)^{*}}{\epsilon}(t-s)}\psi\rangle ds\nonumber\\
&=&\int_{t_{k}}^{t}\int_{\mathbb{R}^{d}}{\rm Tr}[(\bar{\rho}_{s}^{\epsilon}\mathbb{E}^{x}(\bar{v}_{i}^{\epsilon}(s)\otimes \bar{v}_{i}^{\epsilon}(s))-\bar{\rho}_{t_{k}}^{\epsilon}\mathbb{E}^{x}(\bar{v}_{i}^{\epsilon}(t_{k})\otimes \bar{v}_{i}^{\epsilon}(t_{k}))){\rm grad}(e^{-\frac{\gamma(x)^{*}}{\epsilon}(t-s)}\psi)]dxds.
\end{eqnarray*}
Let $g(x)=e^{-\frac{\gamma(x)^{*}}{\epsilon}(t-s)}\psi(x)$, then by the chain rules, 
\begin{eqnarray*}
\frac{\partial}{\partial x_{j}}g(x)=\frac{\partial}{\partial x_{j}}(e^{-\frac{\gamma(x)^{*}}{\epsilon}(t-s)})\psi(x)+e^{-\frac{\gamma(x)^{*}}{\epsilon}(t-s)}\frac{\partial}{\partial x_{j}}\psi(x).
\end{eqnarray*}
By the Theorem $3$ in~\cite{HAB},
\begin{eqnarray*}
\frac{\partial}{\partial x_{j}}(e^{-\frac{\gamma(x)^{*}}{\epsilon}(t-s)})=-\frac{t-s}{\epsilon}\int_{0}^{1}e^{-\theta\frac{\gamma(x)^{*}}{\epsilon}(t-s)}\frac{\partial}{\partial x_{j}}\gamma(x)^{*}e^{-(1-\theta)\frac{\gamma(x)^{*}}{\epsilon}(t-s)}d\theta,
\end{eqnarray*}
then by $(\mathbf{H_{3}})$,
\begin{eqnarray*}
\Big\|\frac{\partial}{\partial x_{j}}g(x)\Big\|\leq C_{L_{\gamma},T}\Big(\frac{1}{\epsilon}+1\Big)\|\psi\|_{Lip},
\end{eqnarray*}
and 
\begin{eqnarray}\label{TDU}
\|{\rm grad}~g(x)\|\leq C_{L_{\gamma},T}\Big(\frac{1}{\epsilon}+1\Big)\|\psi\|_{Lip}.
\end{eqnarray}
Furthermore, by Lemma \ref{lem:3.1} as well as (\ref{TDU}),
\begin{eqnarray}\label{Tra1}
&&\Big|\int_{\mathbb{R}^{d}}{\rm Tr}[(\bar{\rho}_{s}^{\epsilon}\mathbb{E}^{x}(\bar{v}_{i}^{\epsilon}(s)\otimes \bar{v}_{i}^{\epsilon}(s)){\rm grad}~g(x)]dx\Big|\nonumber\\
&=&\frac{1}{\epsilon}\Big|{\rm Tr}[\mathbb{E}(\mathbb{E}^{\bar{x}_{i}^{\epsilon}(s)}(\epsilon\bar{v}_{i}^{\epsilon}(s)\otimes \bar{v}_{i}^{\epsilon}(s)){\rm grad}~g(\bar{x}_{i}^{\epsilon}(s))]\Big|\nonumber\\
&=&\frac{1}{\epsilon}\Big|{\rm Tr}[\mathbb{E}(J(\bar{x}_{i}^{\epsilon}(s),\bar{\rho}_{s}^{\epsilon}(x))+\epsilon C(\bar{x}_{i}^{\epsilon}(s),\bar{\rho}_{s}^{\epsilon}(x),s)){\rm grad}~g(\bar{x}_{i}^{\epsilon}(s))]\Big|\nonumber\\
&\leq&C_{\lambda_{\gamma},\lambda_{\phi},L_{\sigma},T}\Big(\frac{1}{\epsilon}+\frac{1}{\epsilon^{2}}\Big)(1+\mathbb{E}\|\bar{x}_{i}^{\epsilon}(s)\|^{2})\|\phi\|_{Lip}\nonumber\\
&\leq&C_{T,\bar{x}_{i,0},\bar{v}_{i,0},\lambda_{\gamma},\lambda_{\phi},L_{\sigma}}\Big(\frac{1}{\epsilon}+\frac{1}{\epsilon^{2}}\Big)\|\phi\|_{Lip}.
\end{eqnarray}
Similarly,
\begin{eqnarray}\label{Tra2}
\Big|\int_{\mathbb{R}^{d}}{\rm Tr}[(\bar{\rho}_{t_{k}}^{\epsilon}\mathbb{E}^{x}(\bar{v}_{i}^{\epsilon}(t_{k})\otimes \bar{v}_{i}^{\epsilon}(t_{k})){\rm grad}~g(x)]dx\Big|\leq C_{T,\bar{x}_{i,0},\bar{v}_{i,0},\lambda_{\gamma},\lambda_{\phi},L_{\sigma}}\Big(\frac{1}{\epsilon}+\frac{1}{\epsilon^{2}}\Big)\|\phi\|_{Lip}.\nonumber\\
\end{eqnarray}
Combining (\ref{Tra1}) and (\ref{Tra2}),
\begin{eqnarray}\label{Tra3}
|\Pi_{3}^{\epsilon}(t)|\leq C_{T,\bar{x}_{i,0},\bar{v}_{i,0},\lambda_{\gamma},\lambda_{\phi},L_{\sigma}}\Big(\frac{\delta}{\epsilon}+\frac{\delta}{\epsilon^{2}}\Big)\|\phi\|_{Lip}.
\end{eqnarray}
Thus, by (\ref{equ:3.8}), (\ref{equ:3.9}) and (\ref{Tra3}),
\begin{eqnarray*}
\sup_{0\leq t\leq T}|\langle Z_{t}^{\epsilon},\psi\rangle|&\leq& \frac{\delta}{\epsilon}C_{T,\bar{x}_{i,0},\bar{v}_{i,0},L_{V},L_{K},L_{\phi},\lambda_{\gamma},\lambda_{\phi}}\|\psi\|_{Lip}+\frac{L_{\phi}}{\epsilon}\int_{t_{k}}^{t}\sup_{0\leq u\leq s}|\langle Z_{u}^{\epsilon},\psi\rangle|ds\\
&&+C_{T,\bar{x}_{i,0},\bar{v}_{i,0},\lambda_{\gamma},\lambda_{\phi},L_{\sigma}}\Big(\frac{\delta}{\epsilon}+\frac{\delta}{\epsilon^{2}}\Big)\|\phi\|_{Lip},
\end{eqnarray*}
and Gronwall's inequality yields
\begin{eqnarray*}
\sup_{0\leq t\leq T}|\langle Z_{t}^{\epsilon},\psi\rangle|&\leq&C_{T,\bar{x}_{i,0},\bar{v}_{i,0},\lambda_{\gamma},\lambda_{\phi},L_{\sigma},L_{V},L_{K}}\Big(\frac{\delta}{\epsilon}+\frac{\delta}{\epsilon^{2}}\Big)e^{C_{L_{\phi}}\frac{\delta}{\epsilon}}\|\phi\|_{Lip}.
\end{eqnarray*}
\end{proof}

Next we determine the limit of $\hat{Y}_{t}^{\epsilon}$ as $\epsilon\to0$. For this, we introduce the following equation
\begin{eqnarray}\label{equ:3.38}
\partial_{t}\tilde{Y}_{t}^{\epsilon}&=&-\frac{\gamma(x)+\phi*\bar{\rho}(x)}{\epsilon}\tilde{Y}_{t}^{\epsilon}-\frac{\nabla V(x)+\nabla K*\bar{\rho}(x)}{\epsilon}\bar{\rho}(x)\nonumber\\
&&-\nabla_{x}\cdot[\bar{\rho}(x)\mathbb{E}^{x}(\bar{v}_{i}^{\epsilon}(\tau)\otimes \bar{v}_{i}^{\epsilon}(\tau))],
\end{eqnarray}
where $\bar{\rho}\in\mathcal{P}_{G}$, which is the space consisting of Gaussian measure, $\tau>0$. Then we have the following estimate.
\begin{lemma}\label{YHY}
Under the assumption $(\mathbf{H_{1}})$-$(\mathbf{H_{4}})$, for any $x\in\mathbb{R}^{d}$ and fixed $ \bar{\rho}_{t}^{\epsilon}=\bar{\rho}\in\mathcal{P}_{G}$, $t_{*}>0$, $t_{*}\leq t\leq T$, $\psi\in C_{0}^{2}(\mathbb{R}^{d};\mathbb{R}^{d})$,
\begin{eqnarray*}
|\langle\tilde{Y}_{t}^{\epsilon}-Y^{*},\psi\rangle|&\leq&\epsilon C_{t_{*},L_{V},L_{K},L_{\sigma},\lambda_{\gamma},\lambda_{\phi},\bar{x}_{i,0},\bar{v}_{i,0}}\Big(1+\int_{\mathbb{R}^{d}}\|x\|^{2}\bar{\rho}(x)dx\Big)\|\psi\|_{Lip},
\end{eqnarray*}
where 
\begin{eqnarray}\label{EY*}
Y^{*}(x)&=&-(\gamma(x)+\phi*\bar{\rho}(x))^{-1}(\nabla V(x)+\nabla K*\bar{\rho}(x))\bar{\rho}(x)\nonumber\\
&&-(\gamma(x)+\phi*\bar{\rho}(x))^{-1}\nabla_{x}\cdot[\bar{\rho}(x)J(x,\bar{\rho}(x))].
\end{eqnarray}
\end{lemma}
\begin{proof}
By Lemma \ref{YT} and (\ref{equ:3.38}),
\begin{eqnarray*}
\tilde{Y}_{t}^{\epsilon}(x)&=&e^{-\frac{\gamma(x)+\phi*\bar{\rho}(x)}{\epsilon}t}Y_{0}\nonumber\\
&-&\frac{1}{\epsilon}\int_{0}^{t}e^{-\frac{\gamma(x)+\phi*\bar{\rho}(x)}{\epsilon}(t-s)}(\nabla V(x)+\nabla K*\bar{\rho}(x))\bar{\rho}(x)ds\nonumber\\
&-&\frac{1}{\epsilon}\int_{0}^{t}e^{-\frac{\gamma(x)+\phi*\bar{\rho}(x)}{\epsilon}(t-s)}\nabla_{x}\cdot[\bar{\rho}(x)\mathbb{E}^{x}(\epsilon \bar{v}_{i}^{\epsilon}(\tau)\otimes \bar{v}_{i}^{\epsilon}(\tau))]ds,
\end{eqnarray*}     
then       
\begin{eqnarray}\label{T1}
\langle\tilde{Y}_{t}^{\epsilon},\psi\rangle&=&\big\langle e^{-\frac{\gamma(x)+\phi*\bar{\rho}(x)}{\epsilon}t}Y_{0},\psi\big\rangle\nonumber\\
&&-\frac{1}{\epsilon}\int_{0}^{t}\big\langle e^{-\frac{\gamma(x)+\phi*\bar{\rho}(x)}{\epsilon}(t-s)}(\nabla V(x)+\nabla K*\bar{\rho}(x))\bar{\rho}(x),\psi\big\rangle ds\nonumber\\
&&-\frac{1}{\epsilon}\int_{0}^{t}\big\langle e^{-\frac{\gamma(x)+\phi*\bar{\rho}(x)}{\epsilon}(t-s)}\nabla_{x}\cdot[\bar{\rho}(x)\mathbb{E}^{x}(\epsilon \bar{v}_{i}^{\epsilon}(\tau)\otimes \bar{v}_{i}^{\epsilon}(\tau))],\psi\big\rangle ds\nonumber\\
&\triangleq&{\rm III}_{1}^{\epsilon}(t)+{\rm III}_{2}^{\epsilon}(t)+{\rm III}_{3}^{\epsilon}(t).
\end{eqnarray}
For any fixed $t_{*}>0$, by $(\mathbf{H_{1}})$ and the fact $\sup_{x\geq0}x^{p}e^{-ax}\leq C_{p}a^{-p}$, for any $a>0$, we have
\begin{eqnarray}\label{C0}
|{\rm III}_{1}^{\epsilon}(t)|\leq|\big\langle e^{-\frac{\gamma(x)+\phi*\bar{\rho}(x)}{\epsilon}t}Y_{0},\psi\big\rangle|\leq e^{-\frac{\lambda_{\gamma}+\lambda_{\phi}}{\epsilon}t}|\langle Y_{0},\psi\rangle|\leq \epsilon C_{\lambda_{\gamma},\lambda_{\phi},t_{*}}|\|Y_{0}\||_{L^{1}}\|\phi\|_{Lip}.\nonumber\\
\end{eqnarray}
For ${\rm III}_{2}^{\epsilon}(t)$, direct calculation yields
\begin{eqnarray}\label{F1}
{\rm III}_{2}^{\epsilon}(t)&=&\frac{1}{\epsilon}\int_{t_{k}}^{t}\langle e^{-\frac{\gamma(x)+\phi*\bar{\rho}(x)}{\epsilon}(t-s)}(\nabla V(x)+\nabla K*\bar{\rho}(x))\bar{\rho}(x),\psi\rangle ds\nonumber\\
&=&\langle(\gamma(x)+\phi*\bar{\rho}(x))^{-1}(I_{d\times d}-e^{-\frac{\gamma(x)+\phi*\bar{\rho}(x)}{\epsilon}(t-t_{k})})(\nabla V(x)+\nabla K*\bar{\rho}(x))\bar{\rho}(x),\psi\rangle\nonumber\\
&=&\langle(\gamma(x)+\phi*\bar{\rho}(x))^{-1}(\nabla V(x)+\nabla K*\bar{\rho}(x))\bar{\rho}(x),\psi\rangle\nonumber\\
&&+\langle(\gamma(x)+\phi*\bar{\rho}(x))^{-1}e^{-\frac{\gamma(x)+\phi*\bar{\rho}(x)}{\epsilon}t}(\nabla V(x)+\nabla K*\bar{\rho}(x))\bar{\rho}(x),\psi\rangle,
\end{eqnarray}
Note that by $(\mathbf{H_{1}})$
\begin{eqnarray}\label{VKG}
\int_{\mathbb{R}^{d}}\|(\nabla V(x)+\nabla K*\bar{\rho}(x))\|\bar{\rho}(x)dx&\leq& (L_{V}+2L_{K})\int_{\mathbb{R}^{d}}\|x\|\bar{\rho}(x)dx+\|\nabla K(0)\|\nonumber\\
&\leq&C_{L_{V},L_{K},\|\nabla K(0)\|}\Big(1+\int_{\mathbb{R}^{d}}\|x\|\bar{\rho}(x)dx\Big),\nonumber\\
\end{eqnarray}
then by $(\mathbf{H_{3}})$, $(\mathbf{H_{4}})$,
\begin{eqnarray}\label{GEP}
&&|\langle(\gamma(x)+\phi*\bar{\rho}(x))^{-1}e^{-\frac{\gamma(x)+\phi*\bar{\rho}(x)}{\epsilon}t}(\nabla V(x)+\nabla K*\bar{\rho}(x))\bar{\rho}(x),\psi\rangle|\nonumber\\
&=&\Big|\int_{\mathbb{R}^{d}} (\gamma(x)+\phi*\bar{\rho}(x))^{-1}e^{-\frac{\gamma(x)+\phi*\bar{\rho}(x)}{\epsilon}t}(\nabla V(x)+\nabla K*\bar{\rho}(x))\bar{\rho}(x)\psi(x)dx\Big|\nonumber\\
&\leq&\frac{1}{\lambda_{\gamma}+\lambda_{\phi}}e^{-\frac{\lambda_{\gamma}+\lambda_{\phi}}{\epsilon}t}\|\psi\|_{Lip}\int_{\mathbb{R}^{d}}\|(\nabla V(x)+\nabla K*\bar{\rho}(x))\bar{\rho}(x)\|dx\nonumber\\
&\leq&\epsilon C_{t_{*},\lambda_{\gamma},\lambda_{\phi},L_{V},L_{K},\|\nabla K(0)\|}\Big(1+\int_{\mathbb{R}^{d}}\|x\|\bar{\rho}(x)dx\Big)\|\psi\|_{Lip}.
\end{eqnarray}
For ${\rm III}_{3}^{\epsilon}(t)$, by Lemma \ref{lem:3.1},
\begin{eqnarray}\label{GEP2}
{\rm III}_{3}^{\epsilon}(t)&=&-\langle(\gamma(x)+\phi*\bar{\rho}(x))^{-1}(I-e^{-\frac{\gamma(x)+\phi*\bar{\rho}(x)}{\epsilon}t})\nabla_{x}\cdot[\bar{\rho}(x)\mathbb{E}^{x}(\epsilon \bar{v}_{i}^{\epsilon}(\tau)\otimes \bar{v}_{i}^{\epsilon}(\tau))],\psi\rangle\nonumber\\
&=&-\langle(\gamma(x)+\phi*\bar{\rho}(x))^{-1}(I-e^{-\frac{\gamma(x)+\phi*\bar{\rho}(x)}{\epsilon}t})\nabla_{x}\cdot[\bar{\rho}(x)(J(x,\bar{\rho})+\epsilon C(x,\bar{\rho},\tau))],\psi\rangle\nonumber\\
&=&-\langle(\gamma(x)+\phi*\bar{\rho}(x))^{-1}\nabla_{x}\cdot(\bar{\rho}(x)J(x,\bar{\rho})),\psi\rangle\nonumber\\
&&-\epsilon\langle(\gamma(x)+\phi*\bar{\rho}(x))^{-1}\nabla_{x}\cdot[\bar{\rho}(x)C(x,\bar{\rho},\tau)],\psi\rangle\nonumber\\
&&+\langle(\gamma(x)+\phi*\bar{\rho}(x))^{-1}e^{-\frac{\gamma(x)+\phi*\bar{\rho}(x)}{\epsilon}t}\nabla_{x}\cdot[\bar{\rho}(x)(J(x,\bar{\rho})+\epsilon C(x,\bar{\rho},\tau))],\psi\rangle.
\end{eqnarray}
By the Gaussian property of $\bar{\rho}$ and (\ref{JBU}),
\begin{eqnarray}\label{NRJ}
&&\|\nabla_{x}\cdot(\bar{\rho}(x)J(x,\bar{\rho}))\|\nonumber\\
&\leq&\int_{0}^{\infty}\|\nabla_{x}\cdot(\bar{\rho}(x)e^{-(\gamma(x)+\phi*\bar{\rho}(x))t}\sigma(x)\sigma(x)^{*}e^{-(\gamma(x)+\phi*\bar{\rho}(x))^{*}t})\|dt\nonumber\\
&\leq&C_{L_{\gamma},L_{\sigma}}\int_{0}^{\infty}e^{-2(\lambda_{\gamma}+\lambda_{\phi})t}(1+\|x\|)\bar{\rho}(x)dt\nonumber\\
&\leq&C_{L_{\gamma},L_{\sigma},\lambda_{\gamma},\lambda_{\phi}}(1+\|x\|)\bar{\rho}(x).
\end{eqnarray}
Then by $(\mathbf{H_{3}})$, $(\mathbf{H_{4}})$ and (\ref{NRJ}),
\begin{eqnarray}\label{GEP3}
&&|\langle(\gamma(x)+\phi*\bar{\rho}(x))^{-1}e^{-\frac{\gamma(x)+\phi*\bar{\rho}(x)}{\epsilon}t}\nabla_{x}\cdot[\bar{\rho}(x)(J(x,\bar{\rho})+\epsilon C(x,\bar{\rho},\tau))],\psi\rangle|\nonumber\\
&\leq&\frac{1}{\lambda_{\gamma}+\lambda_{\phi}}e^{-\frac{\lambda_{\gamma}+\lambda_{\phi}}{\epsilon}t}\|\psi\|_{Lip}\int_{\mathbb{R}^{d}}\|\nabla_{x}\cdot(\bar{\rho}(x)J(x,\bar{\rho}))\|dx\nonumber\\
&\leq&\epsilon C_{L_{\gamma},L_{\sigma},\lambda_{\gamma},\lambda_{\phi},t_{*}}\Big(1+\int_{\mathbb{R}^{d}}\|x\|\bar{\rho}(x)dx\Big)\|\psi\|_{Lip},
\end{eqnarray}
and
\begin{eqnarray}\label{GEP4}
&&|\langle(\gamma(x)+\phi*\bar{\rho}(x))^{-1}e^{-\frac{\gamma(x)+\phi*\bar{\rho}(x)}{\epsilon}t}\nabla_{x}\cdot[\bar{\rho}(x)(\epsilon C(x,\bar{\rho},\tau))],\psi\rangle|\nonumber\\
&=&\Big|\int_{\mathbb{R}^{d}}\nabla_{x}\cdot[\bar{\rho}(x)(\epsilon C(x,\bar{\rho},\tau))]e^{-\frac{(\gamma(x)+\phi*\bar{\rho}(x))^{*}}{\epsilon}t}[(\gamma(x)+\phi*\bar{\rho}(x))^{*}]^{-1}\psi(x)dx\Big|\nonumber\\
&=&\Big|\int_{\mathbb{R}^{d}}{\rm Tr}[\epsilon\bar{\rho}(x)C(x,\bar{\rho},\tau){\rm grad}(e^{-\frac{(\gamma(x)+\phi*\bar{\rho}(x))^{*}}{\epsilon}t}[(\gamma(x)+\phi*\bar{\rho}(x))^{*}]^{-1}\psi(x))]dx\Big|\nonumber\\
&\leq&\epsilon C_{L_{V},L_{K},L_{\sigma},\lambda_{\gamma},\lambda_{\phi},\bar{x}_{i,0},\bar{v}_{i,0}}\Big(1+\int_{\mathbb{R}^{d}}\|x\|^{2}\bar{\rho}(x)dx\Big)\|\psi\|_{Lip}.
\end{eqnarray}
Then by (\ref{C0}), (\ref{F1}), (\ref{GEP}), (\ref{GEP2}), (\ref{GEP3}), (\ref{GEP4}),
\begin{eqnarray*}
|\langle \tilde{Y}_{t}^{\epsilon}-Y^{*},\psi\rangle|&\leq&\epsilon C_{\lambda_{\gamma},\lambda_{\phi},t_{*}}|\|Y_{0}\||_{L^{1}}\|\phi\|_{Lip}\nonumber\\
&&+\epsilon C_{t_{*},\lambda_{\gamma},\lambda_{\phi},L_{V},L_{K},\|\nabla K(0)\|}\Big(1+\int_{\mathbb{R}^{d}}\|x\|\bar{\rho}(x)dx\Big)\|\psi\|_{Lip}\nonumber\\
&&+\epsilon C_{L_{V},L_{K},L_{\sigma},\lambda_{\gamma},\lambda_{\phi},\bar{x}_{i,0},\bar{v}_{i,0}}\Big(1+\int_{\mathbb{R}^{d}}\|x\|^{2}\bar{\rho}(x)dx\Big)\|\psi\|_{Lip}\nonumber\\
&\leq&\epsilon C_{t_{*},L_{V},L_{K},L_{\sigma},\lambda_{\gamma},\lambda_{\phi},\bar{x}_{i,0},\bar{v}_{i,0}}\Big(1+\int_{\mathbb{R}^{d}}\|x\|^{2}\bar{\rho}(x)dx\Big)\|\psi\|_{Lip},
\end{eqnarray*}
where 
\begin{eqnarray*}
Y^{*}(x)&=&-(\gamma(x)+\phi*\bar{\rho}(x))^{-1}(\nabla V(x)+\nabla K*\bar{\rho}(x))\bar{\rho}(x)\nonumber\\
&&-(\gamma(x)+\phi*\bar{\rho}(x))^{-1}\nabla_{x}\cdot[\bar{\rho}(x)J(x,\bar{\rho}(x))].
\end{eqnarray*}
\end{proof}
Now we are in the position to prove our main result theorem \ref{main2}. 
Without loss of generality, by the tightness of $\{\bar{x}_{i}^{\epsilon}(t)\}$, we assume $\{\bar{\rho}_{\cdot}^{\epsilon}\}$ converges weakly to $\{\bar{\rho}_{\cdot}\}$ in space $C(0,T;\mathbb{R}^{d})$ as $\epsilon\to0$. (If necessary one can choose a subsequence of $\{\bar{\rho}_{\cdot}^{\epsilon}\}$). Now we determine the equation satisfied by $\bar{\rho}_{\cdot}$. Firstly we rewrite $Y^{*}$ as $Y^{*,\bar{\rho}}$ to emphasize the dependence on $\bar{\rho}$ and $\tilde{Y}_{t}^{\epsilon,\bar{\rho},\tau}$ as the solution to (\ref{equ:3.38}). Then on the interval $[t_{k},t_{k+1}]$, $\{\hat{Y}_{t}^{\epsilon}\}=\{\tilde{Y}_{t}^{\epsilon,\bar{\rho}_{t_{k}}^{\epsilon},t_{k}}\}$. Now, let $\psi=\nabla_{x}\phi,\phi\in C_{0}^{\infty}(\mathbb{R}^{d})$, then
\begin{eqnarray}
\langle\bar{\rho}_{t}^{\epsilon},\phi\rangle&=&\langle\bar{\rho}_{t_{*}}^{\epsilon},\phi\rangle+\int_{t_{*}}^{t}\langle Y_{s}^{\epsilon},\psi\rangle ds=\langle\bar{\rho}_{t_{*}}^{\epsilon},\phi\rangle+\int_{t_{*}}^{t}\langle Y^{*, \bar{\rho}_{s}^{\epsilon}},\psi\rangle ds\nonumber\\
&&+\int_{t_{*}}^{t}\langle Y_{s}^{\epsilon}-\hat{Y}_{s}^{\epsilon},\psi\rangle ds+\int_{t_{*}}^{t}\langle \hat{Y}_{s}^{\epsilon}-Y^{*, \bar{\rho}_{s}^{\epsilon}},\psi\rangle ds\nonumber\\
&=&\langle\bar{\rho}_{t_{*}}^{\epsilon},\phi\rangle+K_{1}+K_{2}+K_{3}.
\end{eqnarray}
 Firstly, by the definition of $Y^{*,\bar{\rho}}$ in (\ref{EY*}), and the tightness of $\{\bar{x}_{i}^{\epsilon}(t)\}$ in space $C(0,T;\mathbb{R}^{d})$, for any sequence of $\{\bar{x}_{i}^{\epsilon}(t)\}$, there exists a subsequence denoted still by $\{\bar{x}_{i}^{\epsilon}(t)\}$ such that $\{\bar{x}_{i}^{\epsilon}(t)\}$ converges in distribution to $x(t)$, $\epsilon\to0$. By Skorohod theorem, we still can construct a new probability space and random variables without changing the distribution such that (here we don't changing the notations) $\{\bar{x}_{i}^{\epsilon}(t)\}$ converges almost surely to $x(t)$. Then by assumptions $(\mathbf{H_{1}})$-$(\mathbf{H_{4}})$ and the dominate convergence theorem, we obtain
\begin{eqnarray}\label{BR}
\langle\bar{\rho}_{t_{*}}^{\epsilon},\phi\rangle\to\langle\bar{\rho}_{t_{*}},\phi\rangle,
\end{eqnarray} 
and
\begin{eqnarray}\label{F4}
K_{1}=\int_{t_{*}}^{t}\langle Y^{*, \bar{\rho}_{s}^{\epsilon}},\psi\rangle ds&\to&\int_{t_{*}}^{t}\langle Y^{*,\bar{\rho}_{s}},\psi\rangle ds.
\end{eqnarray}
Further, by Lemma \ref{Le3},
\begin{eqnarray}\label{L1}
|K_{2}|=\Big|\int_{t_{*}}^{t}\langle Y_{s}^{\epsilon}-\hat{Y}_{s}^{\epsilon},\psi\rangle ds\Big|&\leq& C_{T,\bar{x}_{i,0},\bar{v}_{i,0},\lambda_{\gamma},\lambda_{\phi},L_{\sigma},L_{K},L_{V}}\Big(\frac{\delta}{\epsilon}+\frac{\delta}{\epsilon^{2}}\Big)\|\psi\|_{Lip}.\nonumber\\
\end{eqnarray}
Note that, let $t_{*}\in [t_{i_{0}},t_{i_{0}+1}]$,
\begin{eqnarray}
K_{3}&=&\int_{t_{*}}^{t_{i_{0}+1}}\langle \tilde{Y}_{s}^{\epsilon, \bar{\rho}_{t_{i_{0}}}^{\epsilon},t_{i_{0}}}-Y^{*, \bar{\rho}_{t_{i_{0}}}^{\epsilon}}, \psi\rangle ds+\sum_{k=i_{0}+1}^{\lfloor \frac{t}{\delta}\rfloor-1}\int_{t_{k}}^{t_{k+1}}\langle \tilde{Y}_{t}^{\epsilon, \bar{\rho}_{t_{k}}^{\epsilon},t_{k}}-Y^{*, \bar{\rho}_{t_{k}}^{\epsilon}}, \psi\rangle ds\nonumber\\
&&+\int_{t_{\lfloor \frac{t}{\delta}\rfloor}}^{t}\langle\tilde{Y}_{s}^{\epsilon, \bar{\rho}_{t_{\lfloor \frac{t}{\delta}\rfloor}}^{\epsilon},t_{\lfloor \frac{t}{\delta}\rfloor}}-Y^{*, \bar{\rho}_{t_{\lfloor \frac{t}{\delta}\rfloor}}^{\epsilon}},\psi\rangle ds
+\int_{t_{*}}^{t_{i_{0}+1}}\langle Y^{*, \bar{\rho}_{t_{i_{0}}}^{\epsilon}}-Y^{*, \bar{\rho}_{s}^{\epsilon}}, \psi\rangle ds\nonumber
\end{eqnarray}
\begin{eqnarray}\label{W1}
&&+\sum_{k=i_{0}+1}^{\lfloor \frac{t}{\delta}\rfloor-1}\int_{t_{k}}^{t_{k+1}}\langle Y^{*, \bar{\rho}_{t_{k}}^{\epsilon}}-Y^{*, \bar{\rho}_{s}^{\epsilon}},\psi\rangle ds+\int_{t_{\lfloor \frac{t}{\delta}\rfloor}}^{t}\langle Y^{*, \bar{\rho}_{t_{\lfloor \frac{t}{\delta}\rfloor}}^{\epsilon}}-Y^{*, \bar{\rho}_{s}^{\epsilon}},\psi\rangle ds\nonumber\\
&\triangleq&\sum_{j=1}^{6}K_{3,j}.
\end{eqnarray}
By Lemma \ref{YHY} and Lemma \ref{SUB},
\begin{eqnarray}\label{K123}
&&|K_{3,1}+K_{3,2}+K_{3,3}|\nonumber\\
&\leq&\int_{t_{*}}^{t_{i_{0}+1}}\langle \tilde{Y}_{s}^{\epsilon, \bar{\rho}_{t_{i_{0}}}^{\epsilon},t_{i_{0}}}-Y^{*, \bar{\rho}_{t_{i_{0}}}^{\epsilon}}, \psi\rangle ds+\sum_{k=i_{0}+1}^{\lfloor \frac{t}{\delta}\rfloor}\int_{t_{k}}^{t_{k+1}}\langle \tilde{Y}_{t}^{\epsilon, \bar{\rho}_{t_{k}}^{\epsilon},t_{k}}-Y^{*, \bar{\rho}_{t_{k}}^{\epsilon}}, \psi\rangle ds\nonumber\\
&\leq&\epsilon C_{T,\bar{x}_{i,0},\bar{v}_{i,0},t_{*},L_{V},L_{K},L_{\sigma},\lambda_{\gamma},\lambda_{\phi}}\Big(1+\int_{\mathbb{R}^{d}}\|x\|^{2}\bar{\rho}_{t_{i_{0}}}^{\epsilon}(x)dx\Big)\|\psi\|_{Lip}\nonumber\\
&\leq&\epsilon C_{T,\bar{x}_{i,0},\bar{v}_{i,0},t_{*},L_{V},L_{K},L_{\sigma},\lambda_{\gamma},\lambda_{\phi}}\|\psi\|_{Lip}.
\end{eqnarray}
By (\ref{EY*}),
\begin{eqnarray}\label{K123}
|\langle Y^{*, \bar{\rho}_{t_{k}}^{\epsilon}}-Y^{*, \bar{\rho}_{s}^{\epsilon}},\psi\rangle|&\leq&|\langle(\gamma(x)+\phi*\bar{\rho}_{t_{k}}^{\epsilon}(x))^{-1}(\nabla V(x)+\nabla K*\bar{\rho}_{t_{k}}^{\epsilon}(x))\bar{\rho}_{t_{k}}^{\epsilon}(x),\psi\rangle\nonumber\\
&&-\langle(\gamma(x)+\phi*\bar{\rho}_{s}^{\epsilon}(x))^{-1}(\nabla V(x)+\nabla K*\bar{\rho}_{s}^{\epsilon}(x))\bar{\rho}_{s}^{\epsilon}(x),\psi\rangle|\nonumber\\
&&+|\langle(\gamma(x)+\phi*\bar{\rho}_{t_{k}}^{\epsilon}(x))^{-1}\nabla_{x}\cdot[\bar{\rho}_{t_{k}}^{\epsilon}(x)J(x,\bar{\rho}_{t_{k}}^{\epsilon}(x))],\psi\rangle\nonumber\\
&&-\langle(\gamma(x)+\phi*\bar{\rho}_{s}^{\epsilon}(x))^{-1}\nabla_{x}\cdot[\bar{\rho}_{s}^{\epsilon}(x)J(x,\bar{\rho}_{s}^{\epsilon}(x))],\psi\rangle\nonumber\\
&\triangleq&L_{1}^{\epsilon}(s)+L_{2}^{\epsilon}(s).
\end{eqnarray}
Direct computations yield
\begin{eqnarray}\label{L123}
L_{1}^{\epsilon}(s)&\leq&|\langle[(\gamma(x)+\phi*\bar{\rho}_{t_{k}}^{\epsilon}(x))^{-1}-(\gamma(x)+\phi*\bar{\rho}_{s}^{\epsilon}(x))^{-1}](\nabla V(x)+\nabla K*\bar{\rho}_{t_{k}}^{\epsilon}(x))\bar{\rho}_{t_{k}}^{\epsilon}(x),\psi\rangle|\nonumber\\
&&+|\langle(\gamma(x)+\phi*\bar{\rho}_{s}^{\epsilon}(x))^{-1}[\nabla K*\bar{\rho}_{t_{k}}^{\epsilon}(x)-\nabla K*\bar{\rho}_{s}^{\epsilon}(x)]\bar{\rho}_{t_{k}}^{\epsilon}(x),\psi\rangle|\nonumber\\
&&+|\langle(\gamma(x)+\phi*\bar{\rho}_{s}^{\epsilon}(x))^{-1}(\nabla V(x)+\nabla K*\bar{\rho}_{s}^{\epsilon}(x))(\bar{\rho}_{t_{k}}^{\epsilon}(x)-\bar{\rho}_{s}^{\epsilon}(x)),\psi\rangle|\nonumber\\
&=&L_{1,1}^{\epsilon}(s)+L_{1,2}^{\epsilon}(s)+L_{1,3}^{\epsilon}(s).
\end{eqnarray}
For $L_{1,1}^{\epsilon}(s)$, by the Theorem $3$ in~\cite[p177]{MN}, let $A(t)=\gamma(x)+\phi*\bar{\rho}_{t}^{\epsilon}(x)$, then
\begin{eqnarray*}
\frac{dA(t)^{-1}}{dt}=-A(t)^{-1}\frac{dA(t)}{dt}A(t)^{-1}.
\end{eqnarray*}
By the mean value theorem,
\begin{eqnarray*}
A(t_{k})^{-1}-A(s)^{-1}=\frac{dA(u)^{-1}}{dt}(t_{k}-s)=-A(u)^{-1}\frac{dA(u)}{dt}A(u)^{-1}(t_{k}-s),\quad u\in(t_{k},s),
\end{eqnarray*}
that is,
\begin{eqnarray*}
&&(\gamma(x)+\phi*\bar{\rho}_{t_{k}}^{\epsilon}(x))^{-1}-(\gamma(x)+\phi*\bar{\rho}_{s}^{\epsilon}(x))^{-1}\nonumber\\
&=&(\gamma(x)+\phi*\bar{\rho}_{u}^{\epsilon}(x))^{-1}\mathbb{E}^{j}[\phi'(x-\bar{x}_{j}^{\epsilon}(u))(s-t_{k})\bar{v}_{j}^{\epsilon}(u)](\gamma(x)+\phi*\bar{\rho}_{t_{k}}^{\epsilon}(x))^{-1}.
\end{eqnarray*}
Then by (\ref{VKG}), $(\mathbf{H_{3}})$, $(\mathbf{H_{4}})$, Lemma \ref{EWJU} and Lemma \ref{SUB},
\begin{eqnarray}\label{L11}
L_{1,1}^{\epsilon}(s)&\leq&\|\psi\|_{Lip} \int_{\mathbb{R}^{d}}\|(\gamma(x)+\phi*\bar{\rho}_{u}^{\epsilon}(x))^{-1}\|^{2}\|\mathbb{E}^{j}[\phi'(x-\bar{x}_{j}^{\epsilon}(u))(s-t_{k})\bar{v}_{j}^{\epsilon}(u)]\|\nonumber\\
&&\cdot\|\nabla V(x)+\nabla K*\bar{\rho}_{t_{k}}^{\epsilon}(x)\|\bar{\rho}_{t_{k}}^{\epsilon}(x)dx\nonumber\\
&\leq&\frac{\delta}{\epsilon}C_{T,\bar{x}_{i,0},\bar{v}_{i,0},t_{*},L_{V},L_{K},L_{\sigma},\lambda_{\gamma},\lambda_{\phi}}\|\psi\|_{Lip}.
\end{eqnarray}
For $L_{1,2}^{\epsilon}(s)$, by $(\mathbf{H_{2}})$-$(\mathbf{H_{4}})$ and Lemma \ref{SC},
\begin{eqnarray}\label{L12Z}
&&L_{1,2}^{\epsilon}(s)\nonumber\\
&=&\Big|\int_{\mathbb{R}^{d}}(\gamma(x)+\phi*\bar{\rho}_{s}^{\epsilon}(x))^{-1}\mathbb{E}[\nabla K(x-\bar{x}_{i}^{\epsilon}(t_{k}))-\nabla K(x-\bar{x}_{i}^{\epsilon}(s))]\bar{\rho}_{t_{k}}^{\epsilon}(x)\psi(x)dx\Big|\nonumber\\
&\leq&\frac{L_{K}}{\lambda_{\gamma}+\lambda_{\phi}}\mathbb{E}\|\bar{x}_{i}^{\epsilon}(t_{k})-\bar{x}_{i}^{\epsilon}(s)\|\|\psi\|_{Lip}\nonumber\\
&\leq&\sqrt{\delta}C_{T,\bar{x}_{i,0},\bar{v}_{i,0},t_{*},L_{V},L_{K},L_{\sigma},\lambda_{\gamma},\lambda_{\phi}}\|\psi\|_{Lip}.
\end{eqnarray}
For $L_{1,3}^{\epsilon}(s)$, we have
\begin{eqnarray}\label{L13Z}
&&L_{1,3}^{\epsilon}(s)\nonumber\\
&=&\Big|\int_{\mathbb{R}^{d}}(\gamma(x)+\phi*\bar{\rho}_{s}^{\epsilon}(x))^{-1}(\nabla V(x)+\nabla K*\bar{\rho}_{s}^{\epsilon}(x))(\bar{\rho}_{t_{k}}^{\epsilon}(x)-\bar{\rho}_{s}^{\epsilon}(x))\psi(x)dx\Big|\nonumber\\
&=&|\mathbb{E}[(\gamma(\bar{x}_{i}^{\epsilon}(t_{k}))+\phi*\bar{\rho}_{s}^{\epsilon}(\bar{x}_{i}^{\epsilon}(t_{k})))^{-1}(\nabla V(\bar{x}_{i}^{\epsilon}(t_{k}))+\nabla K*\bar{\rho}_{s}^{\epsilon}(\bar{x}_{i}^{\epsilon}(t_{k})))\psi(\bar{x}_{i}^{\epsilon}(t_{k}))]\nonumber\\
&&-\mathbb{E}[(\gamma(\bar{x}_{i}^{\epsilon}(s))+\phi*\bar{\rho}_{s}^{\epsilon}(\bar{x}_{i}^{\epsilon}(s)))^{-1}(\nabla V(\bar{x}_{i}^{\epsilon}(s))+\nabla K*\bar{\rho}_{s}^{\epsilon}(\bar{x}_{i}^{\epsilon}(s)))\psi(\bar{x}_{i}^{\epsilon}(s))]|\nonumber\\
&\leq&|\mathbb{E}(\gamma(\bar{x}_{i}^{\epsilon}(t_{k}))+\phi*\bar{\rho}_{s}^{\epsilon}(\bar{x}_{i}^{\epsilon}(t_{k})))^{-1}(\nabla V(\bar{x}_{i}^{\epsilon}(t_{k}))+\nabla K*\bar{\rho}_{s}^{\epsilon}(\bar{x}_{i}^{\epsilon}(t_{k})))\psi(\bar{x}_{i}^{\epsilon}(t_{k}))\nonumber\\
&&-\mathbb{E}(\gamma(\bar{x}_{i}^{\epsilon}(s))+\phi*\bar{\rho}_{s}^{\epsilon}(\bar{x}_{i}^{\epsilon}(s)))^{-1}(\nabla V(\bar{x}_{i}^{\epsilon}(t_{k}))+\nabla K*\bar{\rho}_{s}^{\epsilon}(\bar{x}_{i}^{\epsilon}(t_{k})))\psi(\bar{x}_{i}^{\epsilon}(t_{k}))|\nonumber\\
&&+|\mathbb{E}(\gamma(\bar{x}_{i}^{\epsilon}(s))+\phi*\bar{\rho}_{s}^{\epsilon}(\bar{x}_{i}^{\epsilon}(s)))^{-1}(\nabla V(\bar{x}_{i}^{\epsilon}(t_{k}))+\nabla K*\bar{\rho}_{s}^{\epsilon}(\bar{x}_{i}^{\epsilon}(t_{k})))\psi(\bar{x}_{i}^{\epsilon}(t_{k}))\nonumber\\
&&-\mathbb{E}(\gamma(\bar{x}_{i}^{\epsilon}(s))+\phi*\bar{\rho}_{s}^{\epsilon}(\bar{x}_{i}^{\epsilon}(s)))^{-1}(\nabla V(\bar{x}_{i}^{\epsilon}(s))+\nabla K*\bar{\rho}_{s}^{\epsilon}(\bar{x}_{i}^{\epsilon}(s)))\psi(\bar{x}_{i}^{\epsilon}(t_{k}))|\nonumber\\
&&+|\mathbb{E}(\gamma(\bar{x}_{i}^{\epsilon}(s))+\phi*\bar{\rho}_{s}^{\epsilon}(\bar{x}_{i}^{\epsilon}(s)))^{-1}(\nabla V(\bar{x}_{i}^{\epsilon}(s))+\nabla K*\bar{\rho}_{s}^{\epsilon}(\bar{x}_{i}^{\epsilon}(s)))\psi(\bar{x}_{i}^{\epsilon}(t_{k}))\nonumber\\
&&-\mathbb{E}(\gamma(\bar{x}_{i}^{\epsilon}(s))+\phi*\bar{\rho}_{s}^{\epsilon}(\bar{x}_{i}^{\epsilon}(s)))^{-1}(\nabla V(\bar{x}_{i}^{\epsilon}(s))+\nabla K*\bar{\rho}_{s}^{\epsilon}(\bar{x}_{i}^{\epsilon}(s)))\psi(\bar{x}_{i}^{\epsilon}(s))|\nonumber\\
&\triangleq&L_{1,3,1}^{\epsilon}(s)+L_{1,3,2}^{\epsilon}(s)+L_{1,3,3}^{\epsilon}(s).
\end{eqnarray}
For $L_{1,3,1}^{\epsilon}(s)$, let $A(t)=\gamma(\bar{x}_{i}^{\epsilon}(t))+\phi*\bar{\rho}_{s}^{\epsilon}(\bar{x}_{i}^{\epsilon}(t))$, then carrying out the same proof as $L_{1,1}^{\epsilon}(s)$ and together with Lemma \ref{EWJU},
\begin{eqnarray}
L_{1,3,1}^{\epsilon}(s)\leq \frac{\delta}{\sqrt{\epsilon}}C_{T,\bar{x}_{i,0},\bar{v}_{i,0},t_{*},L_{V},L_{K},L_{\sigma},\lambda_{\gamma},\lambda_{\phi}}\|\psi\|_{Lip}.
\end{eqnarray}
For $L_{1,3,2}^{\epsilon}(s)$, by $(\mathbf{H_{1}})$-$(\mathbf{H_{4}})$ and Lemma \ref{SC},
\begin{eqnarray}
&&L_{1,3,2}^{\epsilon}(s)\nonumber\\
&\leq&\frac{\|\psi\|_{Lip}}{\lambda_{\gamma}+\lambda_{\phi}}\mathbb{E}[\|\nabla V(\bar{x}_{i}^{\epsilon}(t_{k}))-\nabla V(\bar{x}_{i}^{\epsilon}(s))\|+\|\nabla K*\bar{\rho}_{s}^{\epsilon}(\bar{x}_{i}^{\epsilon}(t_{k}))-\nabla K*\bar{\rho}_{s}^{\epsilon}(\bar{x}_{i}^{\epsilon}(s))\|]\nonumber\\
&\leq&C_{\lambda_{\gamma},\lambda_{\phi},L_{V},L_{K}}\|\psi\|_{Lip}\sqrt{\delta}.
\end{eqnarray}
Similarly, for $L_{1,3,3}^{\epsilon}(s)$,
\begin{eqnarray}\label{L133}
L_{1,3,3}^{\epsilon}(s)\leq C_{T,\bar{x}_{i,0},\bar{v}_{i,0},t_{*},\lambda_{\gamma},\lambda_{\phi},L_{V},L_{K}}\|\psi\|_{Lip}\sqrt{\delta}.
\end{eqnarray}
Thus, by (\ref{L13Z})-(\ref{L133}),
\begin{eqnarray}\label{L13ZZ}
L_{1,3}^{\epsilon}(s)\leq\Big(\sqrt{\epsilon}+\frac{\delta}{\sqrt{\epsilon}}\Big)C_{T,\bar{x}_{i,0},\bar{v}_{i,0},t_{*},\lambda_{\gamma},\lambda_{\phi},L_{V},L_{K}}\|\psi\|_{Lip}.
\end{eqnarray}
Furthermore, by (\ref{L123}), (\ref{L11}), (\ref{L12Z}) and (\ref{L13ZZ}),
\begin{eqnarray}\label{L1ZZZ}
L_{1}^{\epsilon}(s)\leq\Big(\sqrt{\epsilon}+\frac{\delta}{\sqrt{\epsilon}}+\frac{\delta}{\epsilon}\Big)C_{T,\bar{x}_{i,0},\bar{v}_{i,0},t_{*},\lambda_{\gamma},\lambda_{\phi},L_{V},L_{K}}\|\psi\|_{Lip}.
\end{eqnarray}
For $L_{2}^{\epsilon}(s)$,
\begin{eqnarray}\label{L2Z}
L_{2}^{\epsilon}(s)&=&|\langle(\gamma(x)+\phi*\bar{\rho}_{t_{k}}^{\epsilon}(x))^{-1}\nabla_{x}\cdot[\bar{\rho}_{t_{k}}^{\epsilon}(x)J(x,\bar{\rho}_{t_{k}}^{\epsilon}(x))],\psi\rangle\nonumber\\
&&-\langle[(\gamma(x)+\phi*\bar{\rho}_{s}^{\epsilon}(x))^{-1}\nabla_{x}\cdot[\bar{\rho}_{s}^{\epsilon}(x)J(x,\bar{\rho}_{s}^{\epsilon}(x))],\psi\rangle|\nonumber\\
&\leq&|\langle[(\gamma(x)+\phi*\bar{\rho}_{t_{k}}^{\epsilon}(x))^{-1}-(\gamma(x)+\phi*\bar{\rho}_{s}^{\epsilon}(x))^{-1}]\nabla_{x}\cdot[\bar{\rho}_{t_{k}}^{\epsilon}(x)J(x,\bar{\rho}_{t_{k}}^{\epsilon}(x))],\psi\rangle\nonumber\\
&&+|\langle(\gamma(x)+\phi*\bar{\rho}_{s}^{\epsilon}(x))^{-1}\nabla_{x}\cdot[\bar{\rho}_{t_{k}}^{\epsilon}(x)J(x,\bar{\rho}_{t_{k}}^{\epsilon}(x))-\bar{\rho}_{s}^{\epsilon}(x)J(x,\bar{\rho}_{s}^{\epsilon}(x))],\psi\rangle|\nonumber\\
&\triangleq&L_{2,1}^{\epsilon}(s)+L_{2,2}^{\epsilon}(s).
\end{eqnarray}
For $L_{2,1}^{\epsilon}(s)$, Similar proof as (\ref{L11}) yields
\begin{eqnarray}\label{L2Z1}
L_{2,1}^{\epsilon}(s)&\leq& \frac{\delta}{\epsilon}C_{\lambda_{\gamma},\lambda_{\phi},L_{V},L_{K},L_{\phi}}\|\psi\|_{Lip}\mathbb{E}\|\sqrt{\epsilon}\bar{v}_{i}^{\epsilon}(t)\|^{2}\int_{\mathbb{R}^{d}}(1+\|x\|)\bar{\rho}_{t_{k}}^{\epsilon}(x)dx\nonumber\\
&\leq&\frac{\delta}{\epsilon}C_{T,\bar{x}_{i,0},\bar{v}_{i,0},\lambda_{\gamma},\lambda_{\phi},L_{V},L_{K},L_{\sigma}}\|\psi\|_{Lip}.
\end{eqnarray}
For $L_{2,2}^{\epsilon}(s)$, by Lemma \ref{YL},
\begin{eqnarray}\label{L22Z}
&&L_{2,2}^{\epsilon}(s)\nonumber\\
&=&\Big|\int_{\mathbb{R}^{d}}\nabla_{x}\cdot[\bar{\rho}_{t_{k}}^{\epsilon}(x)J(x,\bar{\rho}_{t_{k}}^{\epsilon}(x))-\bar{\rho}_{s}^{\epsilon}(x)J(x,\bar{\rho}_{s}^{\epsilon}(x))][(\gamma(x)+\phi*\bar{\rho}_{s}^{\epsilon}(x))^{*}]^{-1}dx\Big|\nonumber\\
&=&\Big|\int_{\mathbb{R}^{d}}{\rm Tr}[(\bar{\rho}_{t_{k}}^{\epsilon}(x)J(x,\bar{\rho}_{t_{k}}^{\epsilon}(x))-\bar{\rho}_{s}^{\epsilon}(x)J(x,\bar{\rho}_{s}^{\epsilon}(x))){\rm grad}(((\gamma(x)+\phi*\bar{\rho}_{s}^{\epsilon}(x))^{*})^{-1}\psi(x))]dx\Big|.\nonumber\\
\end{eqnarray}
Let $g(x)=((\gamma(x)+\phi*\bar{\rho}_{s}^{\epsilon}(x))^{*})^{-1}\psi(x)$, then
\begin{eqnarray*}
\frac{\partial}{\partial x_{j}}g(x)=\frac{\partial}{\partial x_{j}}(((\gamma(x)+\phi*\bar{\rho}_{s}^{\epsilon}(x))^{*})^{-1})\psi(x)+((\gamma(x)+\phi*\bar{\rho}_{s}^{\epsilon}(x))^{*})^{-1}\frac{\partial}{\partial x_{j}}\psi(x).
\end{eqnarray*}
Let $A(x)=(\gamma(x)+\phi*\bar{\rho}_{s}^{\epsilon}(x))^{*}$, then
\begin{eqnarray*}
\frac{dA(t)^{-1}}{dt}=-A(t)^{-1}\frac{dA(t)}{dt}A(t)^{-1},
\end{eqnarray*}
and 
\begin{eqnarray*}
&&\frac{\partial}{\partial x_{j}}(((\gamma(x)+\phi*\bar{\rho}_{s}^{\epsilon}(x))^{*})^{-1})\\
&=&-((\gamma(x)+\phi*\bar{\rho}_{s}^{\epsilon}(x))^{*})^{-1}\Big[\frac{\partial}{\partial x_{j}}\gamma(x)+\frac{\partial}{\partial x_{j}}(\phi*\bar{\rho}_{s}^{\epsilon}(x))\Big]((\gamma(x)+\phi*\bar{\rho}_{s}^{\epsilon}(x))^{*})^{-1}.
\end{eqnarray*}
Then by $(\mathbf{H_{3}})$ and $(\mathbf{H_{4}})$,
\begin{eqnarray*}
\Big\|\frac{\partial}{\partial x_{j}}g(x)\Big\|\leq C_{\lambda_{\gamma},\lambda_{\phi},L_{\gamma},L_{\phi}}\|\psi\|_{Lip},
\end{eqnarray*}
and
\begin{eqnarray*}
\|{\rm grad}~g(x)\|\leq C_{\lambda_{\gamma},\lambda_{\phi},L_{\gamma},L_{\phi}}\|\psi\|_{Lip}.
\end{eqnarray*}
By (\ref{L22Z}), we have
\begin{eqnarray*}
&&\Big|\int_{\mathbb{R}^{d}}{\rm Tr}[(\bar{\rho}_{t_{k}}^{\epsilon}(x)J(x,\bar{\rho}_{t_{k}}^{\epsilon}(x))-\bar{\rho}_{s}^{\epsilon}(x)J(x,\bar{\rho}_{s}^{\epsilon}(x))){\rm grad}~(g(x))]dx\Big|\nonumber\\
&\leq&C_{\lambda_{\gamma},\lambda_{\phi},L_{\gamma},L_{\phi}}\|\psi\|_{Lip}\int_{\mathbb{R}^{d}}\|\bar{\rho}_{t_{k}}^{\epsilon}(x)J(x,\bar{\rho}_{t_{k}}^{\epsilon}(x))-\bar{\rho}_{s}^{\epsilon}(x)J(x,\bar{\rho}_{s}^{\epsilon}(x))\|dx\nonumber\\
&\leq&C_{\lambda_{\gamma},\lambda_{\phi},L_{\gamma},L_{\phi}}\|\psi\|_{Lip}\Big[\int_{\mathbb{R}^{d}}\|(\bar{\rho}_{t_{k}}^{\epsilon}(x)-\bar{\rho}_{s}^{\epsilon}(x))J(x,\bar{\rho}_{t_{k}}^{\epsilon}(x))\|dx\\
&&+\int_{\mathbb{R}^{d}}\|J(x,\bar{\rho}_{t_{k}}^{\epsilon}(x))-J(x,\bar{\rho}_{s}^{\epsilon}(x))\|\bar{\rho}_{s}^{\epsilon}(x)dx\Big]\nonumber\\
&\triangleq&L_{2,2,1}^{\epsilon}(s)+L_{2,2,2}^{\epsilon}(s).
\end{eqnarray*}
For $L_{2,2,1}^{\epsilon}(s)$, by (\ref{JBU}) and Lemma \ref{SC}, we have
\begin{eqnarray}\label{L22Z1}
L_{2,2,1}^{\epsilon}(s)\leq \sqrt{\delta}C_{\lambda_{\gamma},\lambda_{\phi},L_{\sigma},L_{\gamma},L_{\phi}}\|\psi\|_{Lip}.
\end{eqnarray}
Similarly, 
\begin{eqnarray}\label{L22Z2}
L_{2,2,1}^{\epsilon}(s)\leq\sqrt{\delta}C_{\lambda_{\gamma},\lambda_{\phi},L_{\sigma},L_{\gamma},L_{\phi}}\|\psi\|_{Lip}.
\end{eqnarray}
Combining (\ref{L2Z1}), (\ref{L2Z}) with (\ref{L22Z1}) and (\ref{L22Z2}),
\begin{eqnarray}
L_{2}^{\epsilon}(s)\leq\Big(\frac{\delta}{\epsilon}+\sqrt{\delta}\Big)C_{T,\bar{x}_{i,0},\bar{v}_{i,0},\lambda_{\gamma},\lambda_{\phi},L_{V},L_{K},L_{\sigma}}\|\psi\|_{Lip}.
\end{eqnarray}
This together with (\ref{K123}) and (\ref{L1ZZZ}),
\begin{eqnarray}
|\langle Y^{*, \bar{\rho}_{t_{k}}^{\epsilon}}-Y^{*, \bar{\rho}_{s}^{\epsilon}},\psi\rangle|&\leq&\Big(\sqrt{\delta}+\frac{\delta}{\sqrt{\epsilon}}+\frac{\delta}{\epsilon}\Big)C_{T,\bar{x}_{i,0},\bar{v}_{i,0},\lambda_{\gamma},\lambda_{\phi},L_{V},L_{K},L_{\sigma}}\|\psi\|_{Lip}.\nonumber\\
\end{eqnarray}
Furthermore, by (\ref{W1}),
\begin{eqnarray}
|K_{3,4}+K_{3,5}+K_{3,6}|\leq\Big(\sqrt{\delta}+\frac{\delta}{\sqrt{\epsilon}}+\frac{\delta}{\epsilon}\Big)C_{T,\bar{x}_{i,0},\bar{v}_{i,0},\lambda_{\gamma},\lambda_{\phi},L_{V},L_{K},L_{\sigma}}\|\psi\|_{Lip}.\nonumber\\
\end{eqnarray}
This together with (\ref{K123}) further yields
\begin{eqnarray}
K_{3}\leq\Big(\epsilon+\sqrt{\delta}+\frac{\delta}{\sqrt{\epsilon}}+\frac{\delta}{\epsilon}\Big)C_{T,\bar{x}_{i,0},\bar{v}_{i,0},\lambda_{\gamma},\lambda_{\phi},L_{V},L_{K},L_{\sigma}}\|\psi\|_{Lip}.
\end{eqnarray}
Choosing $\delta=\mathcal{O}(\epsilon^{3})$, 
\begin{eqnarray*}
\langle\bar{\rho}_{t}^{\epsilon},\phi\rangle\to\langle\bar{\rho}_{t_{*}},\phi\rangle+\int_{t_{*}}^{t}\langle Y^{*,\bar{\rho}_{s}},\psi\rangle ds.
\end{eqnarray*}
Note that $\psi=\nabla_{x}\phi$, then
\begin{eqnarray*}
\langle\partial_{t}\bar{\rho}_{t},\phi\rangle=\langle Y^{*,\bar{\rho}_{t}},\psi\rangle=-\langle\nabla_{x}\cdot Y^{*,\bar{\rho}_{t}},\phi\rangle,
\end{eqnarray*}
which is the weak form of
\begin{eqnarray*}
\partial_{t}\bar{\rho}_{t}(x)&=&-\nabla\cdot[(\gamma(x)+\phi*\bar{\rho}_{t}(x))^{-1}(\nabla V(x)+\nabla K*\bar{\rho}_{t}(x))\bar{\rho}_{t}(x)\nonumber\\
&&-(\gamma(x)+\phi*\bar{\rho}_{t}(x))^{-1}\nabla_{x}\cdot[\bar{\rho}_{t}(x)J(x,\bar{\rho}_{t}(x))]].
\end{eqnarray*}


\begin{thebibliography}{99}
\small \setlength{\itemsep}{-.8mm}

\bibitem{BC}F. Bolley, J.A. Ca${\rm\tilde{n}}$izo \& J.A. Carrillo, Stochastic mean-field limit: Non-Lipschitz forces and swarming, {\em Math. Models Methods Appl. Sci.}, 21 (11) (2011) 2179-2210.

\bibitem{BE}R. Bellman, {\em Introduction to Matrix Analysis}, Society for Industrial and Applied Mathematics, 1997.

\bibitem{BBH}M. Behr, P. Benner \& J. Heiland, Solution formulas for differential Sylvester and Lyapunov equations, {\em Arxiv: 1811.08327v1}.

\bibitem{CC}J.A. Carrillo, Y.P. Choi, Mean-field limit: From particle descriptions to macroscopic equations, {\em Arch. Ration. Mech. Anal.}, 241 (3) (2021) 1529-1573.

\bibitem{CS}F. Cucker, S. Smale, On the mathematics of emergence, {\em Jpn. J. Math.}, 2 (2007) 197-227.

\bibitem{CT}Y.P. Choi, O. Tse, Quantified overdamped limit for kinetic Vlasov-Fokker-Planck equations with singular interaction forces, {\em J. Differ. Equ.}, 330 (2022) 150-207. 

\bibitem{D}D.A. Dawson, Critical dynamics and fluctuations for a mean-field model of cooperative behaviour, {\em J. Stat. Phys.}, 31 (1) (1983) 29-85.

\bibitem{FS}R.C. Fetecau, W. Sun, First-order aggregation models and zero inertial limits, {\em J. Differ. Equa.}, 259 (2015) 6774-6802.

\bibitem{FST}R.C. Fetecau, W. Sun \& C.H. Tan, First-order aggregation models with alignment, {\em Phys. D}, 325 (2016) 146-163.

\bibitem{G} F. Golse, On the dynamics of large particle systems in the mean field limit, in {\em Macroscopic and Large Scale Phenomena: Coarse Graining, Mean Field Limits and Ergodicity}, Lecture Notes in Applied Mathematics and Mechanics, Vol 3 (Springer, Cham, 2016), pp. 1-144.

\bibitem{HA}P. Hartman, {\em Ordinary Differential Equations}, John Wiley \& Sons, New York, 1964.

\bibitem{HAB}H.E. Haber, Notes on the matrix exponential and logarithm. Available on https://api.semanticscholar.org/CorpusID:202680000

\bibitem{HL}S.Y. Ha, J.G. Liu, a simple proof of the Cucker-Smale flocking dynamics and mean-field limit, {\em Comm. Math. Sci.}, 7 (2009) 297-325.

\bibitem{HMVW}S. Hottovy, A. McDaniel, G. Volpe \& J. Wehr, The Smoluchowski-Kramers limit of stochastic differential equations with arbitrary state-dependent friction, {\em Comm. Math. Phys.}, 336 (2015) 1259-1283. 

\bibitem{Hu}H. Huang, Quantitative estimate of the overdamped limit for the Vlasov-Fokker-Planck systems, {\em Partial Differ. Equ. Appl. Math.}, 4 (2021) 100186.

\bibitem{HT}S.Y. Ha, E. Tadmor, From particle to kinetic and hydrodynamic descriptions of flocking, {\em Kinet. Relat. Moldels}, 1 (2008) 415-435.

\bibitem{Ja}P.E. Jabin, Macroscopic limit of Vlasov type equations with friction, {\em Ann. Inst. Henri. Poincar${\rm \acute{e}}$}, 17 (2000) 651-672.

\bibitem{Ka}T.K. Karper, Hydrodynamic limit of the kinetic Cucker-Smale flocking model, {\em Math. Models Methods Appl. Sci.}, 25 (2015) 131-163.

\bibitem{Ke}P. Kelly, {\em Mechanics Lecture Notes Part III: Foundations of Continuum Mechanics}, Available on http://homepages.\\
engineering.auckland.ac.nz/pkel015/SolidMechanicsBooks/index.html

\bibitem{M}S. M${\rm\acute{e}}$l${\rm\acute{e}}$ard, Asymptotic behaviour of some interacting particle systems; Mckean-Vlasov and Boltzmann moldels, in {\em Probabilistic Models For Nonlinear Partial Differential Equations}, Lecture Notes in Mathematics Vol. 1464 (Springer, Berlin, 1996), pp. 42-95.

\bibitem{MN}J.R. Magnus \& H. Neudecker, {\em Matrix Differential Calculus with Applications in Statistics and Econometrics}, John Wiley \& Sons, New York, 2007.

\bibitem{M1}M. Freidlin, Some remarks on the Smoluchowski-Kramers approximation, {\em J. Stat. Phys.}, 117 (2004) 617-634.

\bibitem{Se}S. Serfaty, Mean field limit for Coulomb-type flows, {\em J. Duke Math.}, 169 (2020) 2887-2935.

\bibitem{SMD}J.M. Sancho, M.S. Miguel \& D. D${\rm\ddot{u}}$rr, Adiabatic elimination for systems of Brownian particles with non-constant damping coefficients, {\em J. Stat. Phys.}, 28 (2) (1982) 291-305.

\bibitem{S} A.S. Sznitman, Topics in propagation of chaos, in {\em Ecole d'$\acute{E}$t$\acute{e}$ de Probabilit$\acute{e}$s de Saint-Flour XIX-1989,} Lecture Notes in Mathematics Vol. 1464 (Springer, Berlin, 1991), pp. 165-251.

\bibitem{WLW}W. Wang, G.Y. Lv \& J.L. Wei, Small mass limit in mean field theory for stochastic $N$ particle system, {\em J. Math. Phys.}, 63 (2022) 083302.







\bibitem{B}P. Billingsley, {\em Convergence of Probability Measures}, Wiley, New York, 1999.

\bibitem{CF1}S. Cerrai \& M. Freidlin, On the Smoluchowski--Kramers approximation for a system with an infinite number of degrees of freedom, {\em Probab. Th. Rel. Fields}, 135 (2006) 363-394.

\bibitem{CF3}S. Cerrai \& M. Freidlin, Smoluchowski-Kramers approximation for a general class of SPDEs, {\em J. Evol. Equa.}, 6 (2006) 657-689.

\bibitem{CX}S. Cerrai \& M. Xie, On the small noise limit in the Smoluchowski-Kramers approximation of nonlinear wave equations with variable friction, {\em Arxiv: 2203. 05923v2}, 2022.

\bibitem{CX1}S. Cerrai \& G. Xi, A Smoluchowski-Kramers approximation for an infinite dimensional system with state-dependent damping, {\em Ann. Probab.}, 50 (2022) 874-904.

\bibitem{DW}J. Duan \& W. Wang, {\em Effective Dynamics of Stochastic Differential Equations}, Elsevier, London, 2014.

\bibitem{EK}S. Ethier \& T. Kurtz, {\em Markov Processes: Characterization and Convergence}, Wiley, New York, 1986.

\bibitem{F}M. Freidlin, Some remarks on the Smoluchowski-Kramers approximation, {\em J. Stat. Phys.}, 117 (2004) 617-634.

\bibitem{WRDH}W. Wang, J. Ren, J. Duan \& G. He, Ensemble averaging for dynamical systems under fast oscillating random boundary conditions, {\em Stoch. Anal. Appl.}, 32 (2014) 944-961.

\bibitem{LW}Y. Lv \& W. Wang, Diffusion approximation for nonlinear evolutionary equations with large interaction and fast boundary fluctuation, {\em J. Differ. Equa.}, 266 (2019) 3310-3327.

\bibitem{LW1}Y. Lv \& W. Wang, Smoluchowski-Kramers approximation with state dependent damping and highly random oscillation, {\em Discrete Contin. Dyn. Syst. Ser. B}, 28 (2023) 499-515.

\bibitem{GP}G. Pavliotis, {\em Stochastic Processes and Applications: Diffusion Processes, the Fokker-Planck and Langevin Equations}, Springer, New York, 2014.

\bibitem{HMVW}S. Hottovy, A. McDaniel, G. Volpe \& J. Wehr, The Smoluchowski-Kramers limit of stochastic differential equations with arbitrary state-dependent friction, {\em Commun. Math. Phys.}, 336 (2015) 1259-1283.

\bibitem{K}R. Kubo, The fluctuation-dissipation theorem, {\em Rep. Prog. Phys.}, 29 (1996) 255.

\bibitem{LW1}Y. Lv \& W. Wang, Limiting dynamics for stochastic wave equations, {\em J. Differ. Equa.}, 244 (2008) 1-23.

\bibitem{LW2}Y. Lv \& W. Wang, Diffusion approximation for nonlinear evolutionary equations with large interaction and fast boundary fluctuation, {\em J. Differ. Equa.}, 266 (2019) 3310-3327.

\bibitem{M}M. Metivier, {\em Stochastic Partial Differential Equations in Infinite Dimensional Spaces}, Scuola Normale Superiore, Pisa, 1988.

\bibitem{N}H. Nguyen, The small-mass limit and white-noise limit of an infinite dimensional generalized Langevin equation, {\em J. Stat. Phys.}, 173 (2018) 411-437.

\bibitem{Wa}H. Watanabe, Averaging and fluctuations for parabolic equations with rapidly oscillating random coefficients, {\em Probab. Th. Rel. Fields}, 77 (1988) 359-378.

\bibitem{WR}W. Wang \& A. Roberts, Diffusion approximation for self-similarity of stochastic advection in Burgers' Equation, {\em Commun. Math. Phys.}, 333 (2015) 1287-1316.

\bibitem{RZ}R. Zwanzig, {\em Nonequilibrium Statistical Mechanics}, Oxford University Press, 2001.

\bibitem{SW}C. Shi \& W. Wang, Small mass limit and diffusion approximation for a generalized Langevin equation with infinite number degrees of freedom, {\em J. Differ. Equa.}, 286 (2021) 645-675.

\bibitem{Zi}Y. Zine, Smoluchowski-Kramers approximation for the singular stochastic wave equations in two dimensions, {\em Arxiv: 2206. 08717v1}, 2022.





\end{thebibliography}
\end{document}